\documentclass[a4paper,reqno,11pt]{amsart}
\usepackage{amsmath, amsbsy, amssymb, amsfonts, latexsym, mathrsfs}
\usepackage{stmaryrd}
\usepackage{bm}
\usepackage[margin=26mm]{geometry}

\newtheorem{theorem}{Theorem}[section]

\newtheorem{lemma}[theorem]{Lemma}
\newtheorem{proposition}[theorem]{Proposition}

\makeatletter
\newcommand{\imod}[1]{\allowbreak\mkern4mu({\operator@font mod}\,\,#1)}
\makeatother

\newcommand{\la}{\langle}
\newcommand{\ra}{\rangle}
\newcommand{\ep}{\epsilon}
\newcommand{\Z}{\mathbb{Z}}

\newcommand{\Centralizer}{{\mathbf {C}}}
\newcommand{\Normalizer}{{\mathbf {N}}}

\newcommand{\SL}{{\mathrm {SL}}}
\newcommand{\SU}{{\mathrm {SU}}}

\newcommand{\Alt}{\mathrm{A}}
\newcommand{\SSS}{\mathrm{S}}
\newcommand{\Sym}{\mathrm{Sym}}

\newcommand{\F}{\mathbb{F}}

\newcommand{\M}{\mathcal{M}}
\newcommand{\Ed}{\mathcal{E}}

\newcommand{\normeq}{\trianglelefteqslant}

\newcommand{\OO}{\mathbf{O}}

\newtheorem{Theorem}{Theorem}

\newtheorem{Remark}[Theorem]{Remark}

\renewcommand{\a}{\alpha}

\newcommand{\e}{\epsilon}

 \renewcommand{\to}{\rightarrow}

\newcommand{\leqs}{\leqslant}
\newcommand{\geqs}{\geqslant}

 \newcommand{\vs}{\vspace{2mm}}
\renewcommand{\la}{\langle}
\renewcommand{\ra}{\rangle}

\newtheorem{rem}[theorem]{Remark}

\begin{document}
\title[Derangements of prime power order]{Primitive permutation groups and derangements of prime power order}

\author{Timothy C. Burness}
\email{t.burness@bristol.ac.uk }
\address{T.C. Burness, School of Mathematics, University of Bristol, Bristol BS8 1TW, UK}

\author{Hung P. Tong-Viet}
\email{htongvie@kent.edu}
\address{H.P. Tong-Viet, Department of Mathematical Sciences, Kent State University, Kent, Ohio 44242, USA} 



\date{\today}

\begin{abstract}
Let $G$ be a transitive permutation group on a finite set of size at least $2$. By a well known theorem of Fein, Kantor and Schacher, $G$ contains a derangement of prime power order. In this paper, we study the finite primitive permutation groups with the extremal property that the order of every derangement is an $r$-power, for some fixed prime $r$. First we show that these groups are either almost simple or affine, and we determine all the almost simple groups with this property. We also prove that an affine group $G$ has this property if and only if every two-point stabilizer is an $r$-group. Here the structure of $G$ has been extensively studied in work of Guralnick and Wiegand on the multiplicative structure of Galois field extensions, and in later work of Fleischmann, Lempken and Tiep on $r'$-semiregular pairs.
\end{abstract}

\maketitle

\section{Introduction}\label{s:intro}

Let $G$ be a transitive permutation group on a finite set $\Omega$ of size at least $2$. An element $x \in G$ is a \emph{derangement} if it acts fixed-point-freely on $\Omega$. An easy application of the orbit-counting lemma shows that $G$ contains derangements (this is originally a classical theorem of Jordan \cite{Jordan}), and we will write $\Delta(G)$ for the set of derangements in $G$. Note that if $H$ is a point stabilizer, then $x$ is a derangement if and only if $x^G \cap H$ is empty, where $x^G$ denotes the conjugacy class of $x$ in $G$, so we have 
\begin{equation}\label{e:delta}
\Delta(G) = G \setminus \bigcup_{g \in G}H^g.
\end{equation}
The existence of derangements in transitive permutation groups has interesting applications 
in number theory and topology (see Serre's article \cite{Serre}, for example).

Various extensions of Jordan's theorem on the existence of derangements have been studied in recent years. For example, if $\delta(G) = |\Delta(G)|/|G|$ denotes the proportion of derangements in $G$, then a theorem of Cameron and Cohen \cite{CC} states that $\delta(G) \geqs |\Omega|^{-1}$, with equality if and only if $G$ is sharply $2$-transitive. More recently,  Fulman and Guralnick have established the existence of an absolute constant $\e>0$ such that $\delta(G)>\e$ for any simple transitive group $G$ (see \cite{FG1,FG2,FG3,FG4}). This latter result confirms a conjecture of Boston et al. \cite{Boston} and Shalev. 

The study of derangements with special properties has been another major theme in recent years.  
By a theorem of Fein et al. \cite{FKS}, $\Delta(G)$ contains an element of prime power order (their proof requires the classification of finite simple groups), and this result has important number-theoretic applications. For instance, it implies that the relative Brauer group of any finite extension of global fields is infinite. In most cases, $\Delta(G)$ contains an element of prime order, but there are some exceptions, such as the $3$-transitive action of the smallest Mathieu group ${\rm M}_{11}$ on $12$ points. The transitive permutation groups with this property are called \emph{elusive} groups, and they have been investigated by many authors; see \cite{CGJKKMN,Giudici,GiuKel}, for example.

In this paper, we are interested in the permutation groups with the special property that \emph{every} derangement is an $r$-element (that is, has order a power of $r$) for some fixed prime $r$. 
One of our main motivations stems from a theorem of Isaacs et al. \cite{IKLM}, which describes the finite transitive groups in which every derangement is an involution; by \cite[Theorem A]{IKLM}, such a group is either an elementary abelian $2$-group, or a Frobenius group with kernel an elementary abelian $2$-group. In \cite{BDS}, this result is used to classify the finite groups whose irreducible characters vanish only on involutions. It is natural to consider the analogous problem for odd primes, and more generally for prime powers. As noted in \cite{IKLM}, it is easy to see that such a generalization will involve a wider range of examples. For instance, if $p$ is an odd prime then every derangement in the 
affine group ${\rm ASL}_{2}(p) = {\rm SL}_{2}(p){:}p^2$ (of degree $p^2$) has order $p$ (if $p=2$, the derangements have order $2$ or $4$).

\vs

Our first result is a reduction theorem. 

\begin{Theorem}\label{t:main1}
Let $G$ be a finite primitive permutation group such that every derangement in $G$ is an $r$-element for some fixed prime $r$. Then $G$ is either almost simple or affine.
\end{Theorem}

Our next result, Theorem \ref{t:main2} below, describes all the almost simple primitive groups that arise in Theorem \ref{t:main1}. Notice that in Table \ref{tab:main}, we write ${\rm P}_{1}$ for a maximal parabolic subgroup of ${\rm L}_{2}(q)$ or ${\rm L}_{3}(q)$, which can be defined as the stabilizer of a $1$-dimensional subspace of the natural module (similarly, ${\rm P}_{2}$ is the stabilizer of a $2$-dimensional subspace). In addition, we define
$$\Ed(G) = \{|x| \,:\, x \in \Delta(G)\}.$$

\begin{Theorem}\label{t:main2}
Let $G$ be a finite almost simple primitive permutation group with point stabilizer $H$. Then every derangement in $G$ is an $r$-element for some fixed prime $r$ if and only if $(G,H,r)$ is one of the cases in Table \ref{tab:main}. In particular, every derangement has order $r$ if and only if $|\mathcal{E}(G)|=1$.
\end{Theorem}

\begin{Remark}\label{r:isom}
\emph{Let us make a couple of comments on the cases arising in Table \ref{tab:main}.
\begin{itemize}\addtolength{\itemsep}{0.2\baselineskip}
\item[{\rm (i)}] Firstly, notice that the group $G$ is recorded up to isomorphism. For example, the case 
$(G,H)=(\Alt_6, (\SSS_3 \wr \SSS_2) \cap \Alt_6)$ is listed as $({\rm L}_{2}(9),{\rm P}_{1})$, 
$(G,H) = (\Alt_5,\Alt_4)$ appears as $({\rm L}_{2}(4),{\rm P}_{1})$, and we record $(G,H) = ({\rm L}_{2}(7),\SSS_4)$ as $({\rm L}_{3}(2), {\rm P}_{1})$, etc.
\item[{\rm (ii)}] In the first two rows of the table we have $G = {\rm L}_{3}(q)$ and $H = {\rm P}_{1}$ or ${\rm P}_{2}$. Here $q^2+q+1 \in \{r,3r,3r^2\}$, which implies that either $q=4$, or $q=p^f$ for a prime $p$ and $f$ is a $3$-power (see Lemma \ref{l:nagell}).
\end{itemize}}
\end{Remark}

\renewcommand{\arraystretch}{1.1}
\begin{table}[h]
$$\begin{array}{lllll} \hline
G & H & r & \Ed(G) & \mbox{Conditions} \\ \hline
{\rm L}_{3}(q) & {\rm P}_1, {\rm P}_{2} & r & r & q^2+q+1 = (3,q-1)r \\
& & r & r, r^2 & q^2+q+1 = 3r^2 \\
{\rm \Gamma L}_2(q) & \Normalizer_{G}({\rm D}_{2(q+1)}) & r & r & \mbox{$r=q-1$ Mersenne prime} \\
{\rm \Gamma L}_{2}(8) & \Normalizer_{G}({\rm P}_1), \Normalizer_{G}({\rm D}_{14}) & 3 & 3,9 & \\
{\rm PGL}_{2}(q) & \Normalizer_{G}({\rm P}_{1}) & 2 & 2^i, \, 1 \leqs i \leqs e+1 & \mbox{$q=2^{e+1}-1$ Mersenne prime} \\
{\rm L}_2(q) & {\rm P}_1 & r & r^i, \, 1 \leqs i \leqs e & q=2r^e-1 \\
& {\rm P}_1, {\rm D}_{2(q-1)} & r & r & \mbox{$r=q+1$ Fermat prime} \\
& {\rm D}_{2(q+1)} & r & r & \mbox{$r=q-1$ Mersenne prime} \\
{\rm L}_{2}(8) & {\rm P}_{1}, {\rm D}_{14} & 3 & 3,9 & \\
{\rm M}_{11} & {\rm L}_{2}(11) & 2 & 4,8 & \\
\hline
\end{array}$$
\caption{The cases $(G,H,r)$ in Theorem \ref{t:main2}}
\label{tab:main}
\end{table}
\renewcommand{\arraystretch}{1}

Now let us turn our attention to the affine groups that arise in Theorem \ref{t:main1}. In order to state Theorem \ref{t:main3} below, we need to introduce some additional terminology. Let $\F$ be a field and let $V$ be a finite dimensional vector space over $\F$. Let $H \leqs {\rm GL}(V)$ be a finite group and let $r$ be a prime. Recall that $x \in H$ is an \emph{$r'$-element} if the order of $x$ is indivisible by $r$. Following Fleischmann et al. \cite{FLT}, the pair $(H,V)$ is said to be \emph{$r'$-semiregular} if every nontrivial $r'$-element of $H$ has no fixed points on $V\setminus\{0\}$ (equivalently, no nontrivial $r'$-element of $H$ has eigenvalue $1$ on $V$).

\begin{Theorem}\label{t:main3}
Let $G = HV \leqs {\rm AGL}(V)$ be a finite affine primitive permutation group with point stabilizer $H = G_0$ and socle $V = (\mathbb{Z}_{p})^k$, where $p$ is a prime and $k \geqs 1$. Then every derangement in $G$ is an $r$-element for some fixed prime $r$ if and only if $r=p$ and the pair $(H,V)$ is $r'$-semiregular.
\end{Theorem}

Let $G=HV$ be an affine group as in Theorem \ref{t:main3} and notice that 
$(H,V)$ is $r'$-semiregular if and only if every two-point stabilizer in $G$ is an $r$-group. As a special case, observe that if $G$ is a Frobenius group then every two-point stabilizer is trivial and it is clear that every derangement in $G$ has order $r$. Therefore, it is natural to focus our attention on the non-Frobenius affine groups arising in Theorem \ref{t:main3}, which correspond to $r'$-semiregular pairs $(H,V)$ such that $r$ divides $|H|$. In this situation, Guralnick and Wiegand \cite[Section 4]{GW} obtain detailed information on the structure of $H$, which they use to investigate the multiplicative structure of finite Galois field extensions. Similar results were established in later work of Fleischmann et al. \cite{FLT}. We refer the reader to the end of Section \ref{s:affine} for further details (see Propositions \ref{p:flt1} and \ref{p:flt2}).

\vs

Transitive groups with the property in Theorem \ref{t:main1} arise naturally in several different contexts. For instance, let us recall that the existence of a derangement of prime power order in any finite transitive permutation group implies that the relative Brauer group $B(L/K)$ of any finite extension $L/K$ of global fields is infinite. More precisely, let $L = K(\a)$ be a separable extension of $K$, let $E$ be a Galois closure of $L$ over $K$, and let $\Omega$ be the set of roots in $E$ of the minimal polynomial of $\a$ over $K$. Then the $r$-primary component $B(L/K)_r$ is infinite if and only if the Galois group ${\rm Gal}(E/K)$ contains a derangement of $r$-power order on $\Omega$ (see \cite[Corollary 3]{FKS}). In this situation, it follows that the relative Brauer group $B(L/K)$ has a unique infinite primary component if and only if every derangement in ${\rm Gal}(E/K)$ is an $r$-element for some fixed prime $r$.

In a different direction, our property arises in the study of permutation groups with \emph{bounded movement}. To see the connection, let $G \leqs {\rm Sym}(\Omega)$ be a transitive permutation group of degree $n$ and set 
$$m = \max\{|\Gamma^x \setminus \Gamma| \, :\, \Gamma \subseteq \Omega,\, x \in G\} \in \mathbb{N},$$
where $\Gamma^x = \{\gamma^x \,:\, \gamma \in \Gamma\}$. Following Praeger \cite{Praeger}, we say that $G$ has \emph{movement} $m$. If $G$ is not a $2$-group and $n = \lfloor 2mp/(p-1) \rfloor$, where $p\geqs 5$ is the least odd prime dividing $|G|$, then $p$ divides $n$ and every derangement in $G$ has order $p$ (see \cite[Proposition 4.4]{HKKP}). Moreover, the structure of these groups is described in \cite[Theorem 1.2]{HKKP}. 

Some additional related results are established by Mann and Praeger in \cite{MP}. For instance, \cite[Proposition 2]{MP} states that if $G$ is a transitive $p$-group, where $p=2$ or $3$, then every derangement in $G$ has order $p$ only if $G$ has exponent $p$. It is still not known whether or not the same conclusion holds for \emph{any} prime $p$ (see \cite[p.905]{MP}), although \cite[Proposition 6.1]{HKKP} does show that the exponent of such a group is bounded in terms of $p$ only. 

\begin{Remark}\label{r:prime}
\emph{Let $G = HV \leqs {\rm AGL}(V)$ be a finite affine primitive permutation group as above, and assume that every derangement in $G$ is an $r$-element for some fixed prime $r$. Let $P$ be a Sylow $r$-subgroup of $G$ and set $K = H \cap P$. As explained in Proposition \ref{p:prime}, $P$ is a transitive permutation group on $P/K$ such that $\mathcal{E}(G) = \mathcal{E}(P)$, so $\mathcal{E}(G) = \{r\}$ if and only if $\mathcal{E}(P) = \{r\}$, and we will show that $\mathcal{E}(P) = \{r\}$ if and only if $P$ has exponent $r$ (see Theorem \ref{c:prime}).}
\end{Remark}

There is also a connection between our property and $2$-coverings of abstract groups. First notice that Jordan's theorem on the existence of derangements is equivalent to the well known fact that no finite group $G$ can be expressed as the union of $G$-conjugates of a proper subgroup (see \eqref{e:delta}). However, it may be possible to express $G$ as the union of the $G$-conjugates of two proper subgroups; if $H$ and $K$ are proper subgroups such that 
$$G = \bigcup_{g \in G}H^g \cup \bigcup_{g \in G}K^g,$$
then $G$ is said to be \emph{$2$-coverable} and the pair $(H,K)$ is a \emph{$2$-covering} for $G$. This notion has been widely studied in the context of finite simple groups. For instance, Bubboloni \cite{B} proves that $\Alt_n$ is $2$-coverable if and only if $5 \leqs n \leqs 8$, and similarly ${\rm L}_{n}(q)$ is $2$-coverable if and only if $2 \leqs n \leqs 4$ (see \cite{BL}). We refer the reader to \cite{BLW} and \cite{Pellegrini} for further results in this direction. The connection between $2$-coverable groups and the property in Theorem \ref{t:main1} is transparent. Indeed, if $G$ is a transitive permutation group with point stabilizer $H$, then every derangement in $G$ is an $r$-element (for some fixed prime $r$) if and only if $(H,K)$ is a $2$-covering for $G$, where $K$ is a Sylow $r$-subgroup of $G$. 

\vs

Finally, some words on the organisation of this paper. In Section \ref{s:prel} we record several preliminary results that we will need in the proofs of our main theorems. The proof of Theorem \ref{t:main1} is given in Section \ref{s:red}, and the almost simple groups are handled in Section \ref{s:as}, where we prove Theorem \ref{t:main2}. Finally, in Section \ref{s:affine} we turn to affine groups and we establish Theorem \ref{t:main3}.

\vs

\noindent \emph{Notation.} Our group-theoretic notation is standard, and we adopt the notation of Kleidman and Liebeck \cite{KL} for simple groups. For instance, 
$${\rm PSL}_{n}(q) = {\rm L}_{n}^{+}(q) = {\rm L}_{n}(q),\;\; {\rm PSU}_{n}(q) = {\rm L}_{n}^{-}(q) = {\rm U}_{n}(q).$$
If $G$ is a simple orthogonal group, then we write $G = {\rm P\Omega}_{n}^{\e}(q)$, where $\e=+$ (respectively $-$) if $n$ is even and $G$ has Witt defect $0$ (respectively $1$), and $\e=\circ$ if $n$ is odd (in the latter case, we also write $G = \Omega_n(q)$). Following \cite{KL}, we will sometimes refer to the \emph{type} of a subgroup $H$, which provides an approximate description of the group-theoretic structure of $H$. 

For integers $a$ and $b$, we use $(a,b)$ to denote the greatest common divisor  
of $a$ and $b$. If $p$ is a prime number, then we write $a=a_p \cdot a_{p'}$, where $a_p$ is the largest power of $p$ dividing $a$. Finally, if $X$ is a finite set, then $\pi(X)$ denotes the set of prime divisors of $|X|$.

\vs

\noindent \emph{Acknowledgements.} This work was done while the second author held a position at the CRC 701 within the project C13 `The geometry and combinatorics of groups', and he thanks B. Baumeister and G. Stroth for their assistance. Part of the paper was written during the second author's visit to the School of Mathematics at the University of Bristol and he thanks the University of Bristol for its hospitality. Burness thanks R. Guralnick for helpful comments. Both authors thank an anonymous referee for suggesting several improvements to the paper, including a simplified proof of Proposition \ref{p:alt} and a proof of Theorem  \ref{c:prime}.

\section{Preliminaries}\label{s:prel}

In this section we record several preliminary results that will be useful in the proofs of our main theorems.  Let $H$ be a proper subgroup of a finite group $G$ and set
$$\Delta_H(G) = G \setminus \bigcup_{g \in G}H^g.$$ 
Notice that if $G$ is a transitive permutation group with point stabilizer $H$, then $\Delta(G)=\Delta_H(G)$ is the set of derangements in $G$ (see \eqref{e:delta}). 

It will be convenient to define the following property:
\[\mbox{\emph{Every element in $\Delta_H(G)$ is an $r$-element for some fixed prime $r$.} \label{e:star} \tag{$\star$}}\] 

\begin{lemma}\label{lem:power order}
Let $H$ be a proper subgroup of a finite group $G$. If \eqref{e:star} holds, then 
\begin{itemize}\addtolength{\itemsep}{0.2\baselineskip}
\item[{\rm (i)}] $\pi(G) = \pi(H) \cup \{r\}$; and
\item[{\rm (ii)}] $\Centralizer_G(x)$ is an $r$-group for every $x \in\Delta_H(G)$.
\end{itemize}
\end{lemma}

\begin{proof}
If $s \in \pi(G) \setminus \pi(H)$, then $\Delta_H(G)$ contains an $s$-element, so (i) follows. Now consider (ii). Let $x\in\Delta_H(G)$ and assume $s\neq r$ is a prime divisor of $|\Centralizer_G(x)|$. Let $y\in\Centralizer_G(x)$ with $|y|=s$ and let $z=xy=yx$, so $z^s=x^s$ and $\la x \ra \leqs \la z \ra$. Then $z \in \Delta_H(G)$, but this is incompatible with property \eqref{e:star}.
\end{proof}

\begin{lemma}\label{lem:normal}
Let $H$ be a proper subgroup of a finite group $G$, let $N$ be a normal subgroup of $G$ such that $G=NH$, and let $K$ be a proper subgroup of $N$ containing $H\cap N$. Then $\Delta_K(N)\subseteq \Delta_H(G)$.
\end{lemma}

\begin{proof}
Let $x\in \Delta_K(N)$ and assume that $x\not\in\Delta_H(G)$. Then $x^g\in H$ for some $g\in G$. Since $g\in G=NH$, we may write $g=nh$ for some $n\in N$ and $h\in H$. Then $x^g=(x^n)^h\in H$ which implies that $x^n\in H^{h^{-1}}=H$. Since both $x$ and $n$ are in $N$, we deduce that $x^n \in H\cap N\leqs K$, contradicting the fact that $x\in \Delta_K(N)$.
\end{proof}

\begin{rem}\label{r:pgraph}
\emph{Recall that the \emph{prime graph} (or \emph{Gruenberg-Kegel graph}) of a finite group $G$ is the graph $\Gamma(G)$ with vertex set $\pi(G)$ and the property that two distinct vertices $p$ and $q$ are adjacent if and only if $G$ contains an element of order $pq$. Now, a transitive permutation group $G$ with point stabilizer $H$ has property \eqref{e:star} only if one of the following holds:
\begin{itemize}\addtolength{\itemsep}{0.2\baselineskip}
\item[{\rm (a)}] $r$ is an isolated vertex in $\Gamma(G)$;
\item[{\rm (b)}] $\pi(G)=\pi(H)$.
\end{itemize}
The finite simple groups with a disconnected prime graph are recorded in \cite[Tables 1-3]{KM}, and a similar analysis for almost simple groups is given in \cite{Lucido}. In particular, one could use these results to study the almost simple permutation groups for which (a) holds. Similarly, if $G$ is almost simple and (b) holds, then the possibilities for $G$ and $H$ can be read off from \cite[Corollary 5]{LPS}. However, this is not the approach that we will pursue in this paper.}
\end{rem}

The next result is a special case of \cite[Lemma 3.3]{GMS}.

\begin{lemma}\label{l:gms}
Let $G$ be a finite permutation group and let $N$ be a transitive normal subgroup of $G$ such that $G/N = \la Ng \ra$ is cyclic. Then $Ng \cap \Delta(G)$ is empty if and only if every element of $Ng$ has a unique fixed point.
\end{lemma}

We will also need several number-theoretic lemmas. Given a positive integer $n$ we write $n_2$ for the largest power of $2$ dividing $n$. In addition, recall that $(a,b)$ denotes the greatest common divisor of the positive integers $a$ and $b$. The following result is well known. 

\begin{lemma}\label{l:nt}
Let $q \geqs 2$ be an integer. For all integers $n,m \geqs 1$ we have 
\begin{align*} 
(q^n-1,q^m-1) & = q^{(n,m)}-1 \\
(q^n-1,q^m+1) & = \begin{cases}
 q^{(n,m)}+1 & \mbox{if } 2m_2\leqs n_2\\
 (2,q-1)   & \mbox{otherwise} 
\end{cases} \\
(q^n+1,q^m+1) & = \begin{cases}
 q^{(n,m)}+1 & \mbox{if } m_2 = n_2\\
 (2,q-1)   & \mbox{otherwise}  
\end{cases}
\end{align*}
\end{lemma}

Let $q=p^f$ be a prime power, let $e \geqs 2$ be an integer and let $r$ be a prime dividing $q^e-1$. We say that $r$ is a \emph{primitive prime divisor} (ppd for short) of $q^e-1$ if $r$ does not divide $q^{i}-1$ for all $1 \leqs i <e$. A classical theorem of Zsigmondy \cite{Zsig} states that if $e \geqs 3$ then $q^e-1$ has a primitive prime divisor unless $(q,e)=(2,6)$. Primitive prime divisors also exist when $e=2$, provided $q$ is not a Mersenne prime. Note that if $r$ is a ppd of $q^e-1$ then $r \equiv 1 \imod{e}$. Also note that if $n$ is a positive integer, then $r$ divides $q^n-1$ if and only if $e$ divides $n$. If a pdd of $q^e-1$ exists, then we will write $\ell_{e}(q)$ to denote the largest pdd of $q^e-1$. Note that $\ell_e(q)>e$.

\begin{lemma}\label{lem:primepower}
Let $r,s$ be primes, and let $m,n$ be positive integers. If $r^m+1=s^n$, then one of the following holds:
\begin{itemize}\addtolength{\itemsep}{0.2\baselineskip}
\item[{\rm (i)}] $(r,s,m,n) = (2,3,3,2)$; 
\item[{\rm (ii)}] $(r,n) = (2,1)$, $m$ is a $2$-power and $s=2^m+1$ is a Fermat prime;
\item[{\rm (iii)}] $(s,m) = (2,1)$, $n$ is a prime and $r=2^n-1$ is a Mersenne prime. 
\end{itemize}
\end{lemma}

\begin{proof}
This is a straightforward application of Zsigmondy's theorem \cite{Zsig}. For completeness, we will give the details.

First assume that $m=1$, so $r=s^n-1$ is a prime. If $s$ is odd, then $r$ is even, so $r=2$ and $s^n=3$, which implies that $n=1$ and $s=3$. This case appears in (ii). Now assume $s=2$, so $r=2^n-1$ is prime. It follows that $n$ must also be a prime and thus $r$ is a Mersenne prime. This is (iii).

For the remainder, we may assume that $m \geqs 2$. Notice that $r^{2m}-1=s^n(r^m-1)$. If 
$(m,r)= (3,2)$, then $s^n=2^3+1=3^2$ and thus $(s,n)=(3,2)$ as in (i). Now assume that 
$(m,r)\neq (3,2)$. By Zsigmondy's theorem \cite{Zsig}, the ppd $\ell_{2m}(r)$ exists and divides $r^{2m}-1=s^n(r^m-1)$, but not $r^m-1$, hence $s=\ell_{2m}(r)>2m\geqs 4$. Therefore $s \geqs 5$ is an odd prime and $r^m=s^n-1$ is even, so $r=2$. We now consider three cases.

If $n=1$, then $s=r^m+1=2^m+1$ is an odd prime, which implies that $m$ is a $2$-power as in case (ii). Next assume that $n= 2$. Here $2^m=s^2-1=(s-1)(s+1)$ and thus $s-1=2^a$ and $s+1=2^b$ for some positive integers $a$ and $b$. Then $2^b-2^a=(s+1)-(s-1)=2$ and thus  $2^{b-1}=2^{a-1}+1$, which implies that $(a,b)=(1,2)$, so $s=3$ and thus $m=3$. Therefore,  $(r,s,m,n)=(2,3,3,2)$ as in case (i). Finally, let us assume that $n\geqs 3$. Now 
$2^m=s^n-1$ and Zsigmondy's theorem implies that the ppd $\ell_{n}(s)>n\geqs 3$ exists and divides $2^m$, which is absurd.
\end{proof}

\begin{lemma}\label{lem:numerical} 
Let $q$ be a prime power and let $(a,\ep), (b,\delta) \in \mathbb{N} \times \{\pm 1\}$, where $b>a\geqs 2$ and $(a,\ep) \neq (2,-1)$. Let $N=(q^a+\ep)(q^b+\delta)$. Then one of the following holds:
\begin{itemize}\addtolength{\itemsep}{0.2\baselineskip}
\item[{\rm (i)}] $N$ has two distinct prime divisors that do not divide $q^2-1$;
\item[{\rm (ii)}] $(a,\e) = (2,1)$, $(b,\delta) = (4,-1)$ and $q^2+1 = (2,q-1)r^e$ for some prime $r$ and positive integer $e$;
\item[{\rm (iii)}] $q=3$, $(a,\ep) = (2,1)$ and $(b,\delta) = (3,1)$; 
\item[{\rm (iv)}] $q=2$, $(a,\ep) = (3,1)$ and $2^b+\delta$ is divisible by at most two distinct primes, one of which is $3$;
\item[{\rm (v)}] $q=2$, $a=3$ and $(b,\delta) = (6,-1)$.
\end{itemize}
\end{lemma}

\begin{proof}
There are four cases to consider, according to the possibilities for the pair $(\ep,\delta)$.

First assume that $(\ep,\delta)=(1,1)$. Suppose that neither $(a,q)$ nor $(b,q)$ is equal to $(3,2)$. Then the primitive prime divisors $\ell_{2a}(q)$ and $\ell_{2b}(q)$ exist, and they both divide $N$. Moreover, these primes are distinct since 
$2a<2b$, and neither of them divides $q^2-1$ since $2b>2a\geqs 4$. If $(a,q)=(3,2)$ then $b\geqs 4$, $N = 3^2(2^b+1)$ and either (i) or (iv) holds. If $(b,q)=(3,2)$, then $a=2$, $N=3^2\cdot 5$ and (iii) holds.

Next suppose that $(\ep,\delta)=(-1,-1)$, so $a \geqs 3$. If neither $(a,q)$ nor $(b,q)$ is equal to $(6,2)$, then $N$ is divisible by the distinct primes $\ell_a(q)$ and $\ell_b(q)$, neither of which divide $q^2-1$. If $(a,q)=(6,2)$, then $N=3^2\cdot 7(2^b-1)$ is divisible by $7$ and 
$\ell_b(2)>b \geqs 7$. Finally, suppose that $(b,q)=(6,2)$, so $N=3^2\cdot 7(2^a-1)$ and $3\leqs a\leqs 5$. It is easy to check that (i) holds if $a=4$ or $5$, and that (v) holds if $a=3$.

Now assume that $(\ep,\delta)=(1,-1)$. If $(a,q)=(3,2)$ then (i) or (iv) holds, so we may assume that $(a,q)\neq (3,2)$. If $(b,q)=(6,2)$ then 
$N=3^2\cdot 7(2^a+1)$, $a \in \{2,4,5\}$ and (i) holds. In each of the remaining cases, the primitive prime divisors $\ell_{2a}(q)$ and $\ell_b(q)$ exist, and they divide $N$, but not $q^2-1$. Clearly, if $b\neq 2a$ then these two primes are distinct and (i) holds, so let us assume that $b=2a$, so $N=(q^a+1)^2(q^a-1)$. If $(a,q)=(6,2)$ then (i) holds. If $(a,q)\neq (6,2)$ and $a \geqs 3$ then we can take the primitive prime divisors $\ell_a(q)$ and $\ell_{2a}(q)$, so once again (i) holds. Finally, if $a=2$ and $b=4$ then $N = (q^2-1)(q^2+1)^2$ and either (i) or (ii) holds.

Finally, let us assume that $(\ep,\delta)=(-1,1)$. Here we may assume that $a \geqs 3$. If $(a,q)\neq (6,2)$ then take 
$\ell_a(q)$ and $\ell_{2b}(q)$, otherwise $N=3^2\cdot 7(2^b+1)$ is divisible by $7$ and $\ell_{2b}(2)$. In both cases, (i) holds.
\end{proof}

\begin{lemma}\label{l:ppower}
Let $q$ be a prime power and let $N$ be one of the integers in Table \ref{tab:int}, where $\e=\pm 1$. Then $N$ is a prime power if and only if $(\e,q)$ is one of the cases recorded in the second column of the table.
\end{lemma}

\renewcommand{\arraystretch}{1.1}
\begin{table}[h]
$$\begin{array}{ll} \hline
N & (\e,q) \\ \hline
(q^6-1)/(7,q-\e) & \mbox{none} \\
(q^6-1)/(q-\e)(6,q-\e) & (-,2) \\
(q^5-\e)/(6,q-\e) &  (+,2), (+,3), (+,7), (-,2), (-,5) \\
(q^4-1)/(5,q-\e) & \mbox{none}\\
(q^4-1)/(q-\e)(4,q-\e) & (-,2), (-,3) \\
(q^3-\e)/(4,q-\e) &  (+,2), (+,3), (+,5), (-,2), (-,3) \\
(q^3-1)(q+1)/(5,q-\e) & \mbox{none}\\
\hline
\end{array}$$
\caption{The integers $N$ in Lemma \ref{l:ppower}}
\label{tab:int}
\end{table}
\renewcommand{\arraystretch}{1}

\begin{proof}
This is entirely straightforward. For example, suppose that $N = (q^5-1)/(6,q-1)$. Let $d = (6,q-1)$ and suppose that $N = r^e$ for some prime number $r$ and positive integer $e$. Then $r = \ell_{5}(q)$ and 
$$(q-1)(q^4+q^3+q^2+q+1) = dr^e.$$
Since $r$ does not divide $q-1$, we must have $q-1=d$ and thus $q-1 \in \{1,2,3,6\}$. If $q=4$ then $N = 341 = 11\cdot 31$ is not a prime power, but one checks that $N$ is a prime power if $q \in \{2,3,7\}$. The other cases are very similar.
\end{proof}

We will also need the following result, which follows from a theorem of Nagell \cite{Nagell}.

\begin{lemma}\label{l:nagell}
Let $q = p^f$ be a prime power and let $r$ be a prime.
\begin{itemize}\addtolength{\itemsep}{0.2\baselineskip}
\item[{\rm (i)}] If $e$ is a positive integer such that $q^2+q+1 = r^e$, then $q \not\equiv 1 \imod{3}$ and $e=1$.
\item[{\rm (ii)}] If $e$ is a positive integer such that $q^2+q+1 = 3r^e$, then $q \equiv 1 \imod{3}$ and $e \in \{1,2\}$.
\item[{\rm (iii)}]  If $q^2+q+1  = (3,q-1)r^e$ for some positive integer $e$, then either $(q,r,e) = (4,7,1)$, or $f=3^a$ for some integer $a \geqs 0$.
\end{itemize}
\end{lemma}

\begin{proof} 
Parts (i) and (ii) follow from \cite{Nagell}. For (iii), let $d=(3,q-1)$ and write $f=3^a m$ with $(3,m)=1$ and $a\geqs 0$. We may assume that $q \neq 4$. Seeking a contradiction, suppose that $m>1$. Notice that 
$$r^e=\frac{p^{3^{a+1}m}-1}{d(p^{3^am}-1)}.$$ 
Since $q \neq 4$, the ppd $\ell_{3f}(p)$ exists and divides $q^2+q+1$, so $r=\ell_{3f}(p)$. Let 
$s=\ell_{3^{a+1}}(p)$. Since $f=3^am$ is indivisible by $3^{a+1}$, it follows that $(s,q-1)=1$, so  $s$ does not divide $d(q-1)$ and thus $s$ divides $r^e$, so $r=s$. But $m>1$, so $3f>3^{a+1}$ and thus $r \neq s$. This is a contradiction and the result follows.
\end{proof}

\begin{rem}
\emph{By a theorem of van der Waall \cite{Waall}, the Diophantine equation $x^2+x+1 = 3y^2$ has infinitely many integer solutions; the smallest nontrivial solution is $(x,y) = (313,181)$. Here $x$ and $y$ are both primes, and another solution in the primes is $(x,y) = (2288805793,1321442641)$.}
\end{rem}

\section{A reduction theorem}\label{s:red}

The following theorem reduces the study of primitive permutation groups with property \eqref{e:star} 
to almost simple and affine groups.

\begin{theorem}\label{thm:reduction} 
Let $G\leqs \Sym(\Omega)$ be a primitive permutation group with point stabilizer $H$. If \eqref{e:star} holds, then either
\begin{itemize}\addtolength{\itemsep}{0.2\baselineskip}
\item[{\rm (i)}] $G$ is almost simple; or
\item[{\rm (ii)}] $G=HN$ is an affine group with socle $N \cong (\mathbb{Z}_{r})^k$ for some integer $k \geqs 1$. 
\end{itemize}
Moreover, if (ii) holds and $|H|$ is indivisible by $r$, then $G$ is a Frobenius group with kernel $N$ and complement $H$.
\end{theorem}

\begin{proof} 
Let $N$ be a minimal normal subgroup of $G$, so $N\cong S_1\times S_2\times\cdots\times S_k$, where $S_i\cong S$ for some simple group $S$ and integer $k\geqs 1$. Then $G=HN$ and $N$ is transitive on $\Omega$. Let us assume that \eqref{e:star} holds.

First assume that $H\cap N=1$, so $N$ is regular and every nontrivial element in $N$ is a derangement. If $N$ is abelian, then we are in case (ii). Moreover, if $|H|$ is indivisible by $r$,  then $N$ is a Sylow $r$-subgroup of $G$ and thus $\Delta(G)\subseteq N$. In this situation,  \cite[Lemma 4.1]{BTV} implies that $G$ is a Frobenius group with kernel $N$ and complement $H$. Now, if $N$ is nonabelian then $S$ is a nonabelian simple group and thus $|S|$ is divisible by at least three distinct primes, whence $S$ (and thus $N$) contains derangements of distinct prime orders, which is incompatible with property \eqref{e:star}.

For the remainder, we may assume that $H\cap N$ is nontrivial. It follows that $N \cong S^k$, where $S$ is a nonabelian simple group and $k\geqs 1$.  If $k=1,$ then $G$ is almost simple and (i) holds. Therefore, we may assume that $k\geqs 2$. 
 
Let $T\leqs N$ be a maximal subgroup of $N$ containing $H\cap N$. By Lemma \ref{lem:normal}, we have $\Delta_T(N) \subseteq \Delta_H(G)$. Since $k\geqs 2$, there exist integers $i$ and $j$ such that  $1\leqs i<j \leqs k$ and $L:=S_i\times S_j\not\leqs T$. By relabelling the $S_{\ell}$, if necessary, we may assume that $L = S_1 \times S_2$. Now $L\normeq N$, so $N=TL$ and thus 
\begin{equation}\label{e:deltak}
\Delta_K(L)\subseteq \Delta_T(N) \subseteq\Delta_H(G)=\Delta(G),
\end{equation}
where $K$ is a maximal subgroup of $L$ containing $L\cap T$. Therefore, every derangement of $L=S_1\times S_2$ on the right cosets $L/K$ is an $r$-element.

By \cite[Lemma 1.3]{Thevenaz}, there are essentially two possibilities for $K$; either $K$ is a diagonal subgroup of the form $\{(s,\phi(s)) \,:\, s\in S_1\}$ for some isomorphism $\phi:S_1 \to S_2$, or $K$ is a standard maximal subgroup, i.e., $K=S_1\times K_2$ or $K_1\times S_2$, where $K_i<S_i$ is maximal. In the diagonal case, every element in $L$ of the form $(s,1)$ with $1\neq s\in S_1$ is a derangement on $L/K$. Clearly, this situation cannot arise. Now assume $K$ is a standard maximal subgroup. Without loss of generality, we may assume that $K=K_1\times S_2$, where $K_1$ is maximal in $S_1$. Let $s\in S_1$ be a derangement on $S_1/K_1$ of prime power order, say $p^e$ for some prime $p$ and integer $e\geqs 1$ (such an element exists by the main theorem of \cite{FKS}). Since
$|\pi(S)|\geqs 3$, choose $t\in S_2$ of prime order different from $p$.
Then $(s,t) \in L$ is a derangement on $L/K$ of non-prime power order, so once again we have reached a contradiction.
\end{proof}

\vs

This completes the proof of Theorem \ref{t:main1}. 

\section{Almost simple groups}\label{s:as}

In this section we prove Theorem \ref{t:main2}. We fix the following notation. Let $r$ be a prime and let $G \leqs \Sym(\Omega)$ be an almost simple primitive permutation group with socle $G_0$ and point stabilizer $H$. Set $H_0 = H \cap G_0$ and let $M$ be a maximal subgroup of $G_0$ containing $H_0$. 
As before, let $\Delta(G)$ be the set of derangements in $G$, and let $\Ed(G)$ be the set of orders of elements in $\Delta(G)$. By Lemma \ref{lem:normal}, we have
\begin{equation}\label{e:g0}
\Delta_{M}(G_0) \subseteq \Delta_{H_0}(G_0) \subseteq \Delta_{H}(G) = \Delta(G).
\end{equation}
Recall that if $X$ is a finite set, then $\pi(X)$ denotes the set of prime divisors of $|X|$.

Let us assume that \eqref{e:star} holds, so every derangement in $G$ is an $r$-element, for some fixed prime $r$. Clearly, every derangement of $G_0$ on $\Omega$ is also an $r$-element. Now, if $s \in \pi(G_0) \setminus \pi(M)$ then every nontrivial $s$-element in $G_0$ is a derangement, so $\pi(G_0)=\pi(M)$ or $\pi(M) \cup \{r\}$. In particular, if we set 
$\pi_0:=\pi(G_0)\setminus\pi(M)$, then $|\pi_0|\leqs 1$. 

\subsection{Sporadic groups}\label{ss:sporadic}

\begin{proposition}\label{p:spor}
Theorem \ref{t:main2} holds if $G_0$ is a sporadic group or the Tits group.
\end{proposition}

\begin{proof}
First assume that $G_0$ is not the Monster. The maximal subgroups of $G_0$ are available in \textsf{GAP} \cite{GAP}, and it is easy to identify the cases $(G_0,M)$ with $|\pi_0| \leqs 1$. For the reader's convenience, the cases that arise are listed in Table \ref{Tab:sporadic}. We now consider each of these cases in turn.  With the aid of \textsf{GAP} \cite{GAP}, we can compute the permutation character $\chi=1_M^{G_0}$, and we observe that 
$$\Delta_M(G_0) = \{x \in G_0 \,:\, \chi(x)=0\}.$$ 
In this way, we deduce that property \eqref{e:star} holds if and only if 
$(G_0,M)=({\rm M}_{11},{\rm L}_2(11))$. Here $\pi(M) = \pi(G_0)$, $G = {\rm M}_{11}$, $H = {\rm L}_{2}(11)$ and $\Ed(G) = \{4,8\}$. This case is recorded in Table \ref{tab:main}.

Now assume $G = \mathbb{M}$ is the Monster. As noted in \cite{BW,NW}, there are $44$ conjugacy classes of known maximal subgroups of $\mathbb{M}$ (these subgroups are conveniently listed in \cite[Table 1]{BW}, together with ${\rm L}_2(41)$). Moreover, it is known that any additional maximal subgroup of $\mathbb{M}$ is almost simple with socle ${\rm L}_2(13),{\rm U}_3(4),{\rm U}_3(8)$ or ${}^2{\rm B}_2(8)$. It is routine to check that $|\pi_0|\geqs 2$ in each of these cases.
\end{proof}

\renewcommand{\arraystretch}{1.1}
\begin{table}
$$\begin{array}{lll} \hline
G_0 & M & \pi_0 \\ \hline
    
    {\rm M}_{11}& \Alt_6.2_3,\SSS_5&11\\

                  & {\rm L}_2(11)&-\\
     {\rm M}_{12}& {\rm M}_{11},{\rm L}_{2}(11)&-\\ 
     & \Alt_6.2^2, 2\times\SSS_5&11\\            
                  
       {\rm M}_{22}& \Alt_7, {\rm L}_3(4)&11\\ 
       & {\rm L}_2(11)&7\\
        {\rm M}_{23}& {\rm M}_{22}&23\\    
        {\rm M}_{24}& {\rm M}_{23}&-\\ 
        & {\rm M}_{22}.2 &23\\  
        
         {\rm J}_{2}& {\rm L}_3(2).2, {\rm U}_3(3) &5\\ 
         & 3.\Alt_6.2_2, 2^{1+4}{:}\Alt_5, \Alt_4\times\Alt_5,\Alt_5\times {\rm D}_{10},5^2{:}{\rm D}_{12},\Alt_5&7\\ 
         
         {\rm J}_{3}& {\rm L}_2(16).2&19\\ 
         & {\rm L}_2(19)&17\\ 
         
          {\rm Co}_{1}& 3.{\rm Suz}.2&23\\  
          & {\rm Co}_2,{\rm Co}_3,2^{11}{:}{\rm M}_{24}&13\\
          
          {\rm Co}_{2}& {\rm M}_{23}&-\\  
                    & {\rm McL},{\rm HS}.2,{\rm U}_6(2).2,2^{10}{:}{\rm M}_{22}.2 &23\\  
          {\rm Co}_{3}& {\rm M}_{23}&-\\  
                              & {\rm McL}.2,{\rm HS}&23\\
           {\rm Fi}_{22}& 2.{\rm U}_6(2),2^{10}{:}{\rm M}_{22}&13\\  
                     & \Omega_7(3)&11\\                   
                               
           {\rm Fi}'_{24}& {\rm Fi}_{23}&29\\                   
          {\rm HS}& {\rm M}_{22}&-\\
          & {\rm U}_3(5).2,{\rm L}_3(4).2_1,\SSS_8&11\\ 
           &{\rm M}_{11}&7\\ 
           
           {\rm McL}& {\rm M}_{22}&-\\
                     & {\rm U}_4(3),{\rm U}_3(5),{\rm L}_3(4).2_2, 2.\Alt_8, 2^4{:}\Alt_7&11\\ 
                      &{\rm M}_{11}&7\\ 
                      
           {\rm Suz}&{\rm G}_2(4)&11\\

            {\rm He}& {\rm Sp}_4(4).2&7\\  

            &2^2.{\rm L}_3(4).\SSS_3, 3.\SSS_7&17\\ 
            {\rm HN}&2.{\rm HS}.2,\Alt_{12}&19\\ 
            
             {\rm O'N}&{\rm J}_1&31\\ 
             
              {\rm Ru}&(2^2\times {}^2{\rm B}_2(8)){:}3 &29\\   
              &{\rm L}_2(29)&13\\
              {}^2{\rm F}_4(2)'&{\rm L}_2(25)&-\\ 
              &{\rm L}_3(3).2&5\\  
              &\Alt_6.2^2,5^2{:}4\Alt_4&13\\ 
              \hline  
  \end{array}$$
  \caption{Maximal subgroups of sporadic simple groups, $|\pi_0| \leqs 1$}
  \label{Tab:sporadic}
\end{table}
\renewcommand{\arraystretch}{1}

\subsection{Alternating groups}\label{ss:alt}

\begin{proposition}\label{p:alt}
Theorem \ref{t:main2} holds if $G_0 = \Alt_n$ is an alternating
group.
\end{proposition}

\begin{proof}
If $n< 12$ then the result can be checked directly using \textsf{GAP} \cite{GAP}; the only cases $(G,H)$ with property \eqref{e:star} are the following:
$$(\Alt_6, 3^2{:}4), \, (\Alt_5, {\rm D}_{10}), \, (\Alt_5, \Alt_4), \, (\Alt_5, {\rm S}_{3}),$$
which are recorded in Table \ref{tab:main} as
$$({\rm L}_{2}(9),{\rm P}_{1}),\, ({\rm L}_{2}(4),{\rm D}_{10}),\, ({\rm L}_{2}(4),{\rm P}_{1}),\, ({\rm L}_{2}(4),{\rm D}_{6})$$
respectively (see Remark \ref{r:isom}). For the remainder, we may assume that $n \geqs 12$. 
Seeking a contradiction, let us assume that there is 
a fixed prime $r$ such that every derangement in $G$ is an $r$-element. 

Let $s$ be a prime such that $n/2<s<n-2$ and let $x \in G_0$ be an $s$-cycle (such a prime exists by \emph{Bertrand's postulate}). Since $\Centralizer_G(x)$ is not an $r$-group, Lemma \ref{lem:power order}(ii) implies that $x$ is not a derangement and thus $H$ contains $s$-cycles. By applying a well known theorem of Jordan (see \cite[Theorem 13.9]{Wielandt}), we deduce that $H$ is either intransitive or imprimitive, and we can rule out the latter possibility since $s$ divides $|H|$. Therefore, $H$ is the stabilizer of a $k$-set for some $k$ with $1<k<n/2$.

Suppose $n$ is even and let $x_i \in G_0$ be an element with cycles of length $i$ and $n-i$ for $i \in \{3,5,7\}$. Then at least two of the $x_i$ are derangements, so we have reached a contradiction. Now assume $n$ is odd. An $n$-cycle does not fix a $k$-set, so $n$ must be an $r$-power. Therefore, any element with cycles of length $(n-1)/2$, $(n-1)/2$ and $1$ must fix a $k$-set (since its order is not an $r$-power), so $k=1$ or $(n-1)/2$. It follows that any element with cycles of length $2,3$ and $n-5$ is a derangement, and this final contradiction completes the proof of the proposition.
\end{proof}

\subsection{Exceptional groups}\label{ss:ex}

Now let us assume that $G_0$ is a simple exceptional group of Lie type over $\mathbb{F}_{q}$, where $q=p^f$ and $p$ is a prime. For $x \in G_0$, let $\mathcal{M}(x)$ be the set of maximal subgroups of $G_0$ containing $x$. We will write $\Phi_i$ for the $i$-th cyclotomic polynomial evaluated at $q$, so $q^n - 1 = \prod_{d|n}\Phi_d$. Recall that if $e \geqs 2$ and $q^e-1$ has a primitive prime divisor, then we use the notation $\ell_{e}(q)$ to denote the largest such divisor of $q^e-1$.

\begin{proposition}\label{p:ex}
Theorem \ref{t:main2} holds if $G_0$ is a simple exceptional group of Lie type.
\end{proposition}

\begin{proof} 
Recall the notational set-up introduced at the beginning of Section \ref{s:as}: $H$ is a point stabilizer in $G$, and $H_0 = H \cap G_0$. In view of \eqref{e:g0}, 
in order to show that \eqref{e:star} does not hold we may assume that $G=G_0$. Seeking a contradiction, suppose that every derangement in $G$ is an $r$-element, for some fixed prime $r$. We will consider each possibility for $G$ in turn. 

\vs

\noindent \emph{Case 1.} $G={}^2{\rm B}_2(q)$, with $q=2^{2m+1}$ and $m\geqs 1$.

\vs

Let $\Phi_4'=q+\sqrt{2q}+1$ and $\Phi_4''=q-\sqrt{2q}+1$ (note that $\Phi_4'\Phi_4'' = q^2+1$). By inspecting \cite[Table II]{BPS}, \cite[Table 6]{GM1} and \cite[Table 1]{GM2}, we see that $G$ has two cyclic maximal tori $T_i=\la x_i\ra$, $i=1,2$, of order $\Phi_4'$ and $\Phi_4''$, respectively, such that $|\Normalizer_{G}(T_i)/T_i|=4$, $(|x_1|,|x_2|)=1$ and $\M(x_i)=\{\Normalizer_{G}(T_i)\}$. Since no maximal subgroup of $G$ can contain conjugates of both $x_1$ and $x_2$, it follows that  $x_i^{G}\cap H$ is empty for some $i=1,2$. Therefore, $x_i\in\Delta(G)$ and thus $|x_i|$ is a power of $r$. Let $j=3-i$. Then $|x_j|$ is indivisible by $r$, so $H$ contains a conjugate of $x_j$ and thus $H=\Normalizer_{G}(T_j)$ is the only possibility (up to conjugacy). Now $G$ has a cyclic maximal torus of order $q-1$, so let $x \in G$ be an element of order $q-1\geqs 7$. Since $|H|$ is indivisible by $q-1$, it follows that $x \in \Delta(G)$. But $r$ does not divide $q-1$, so we have reached a contradiction.

\vs

\noindent \emph{Case 2.} $G = {}^2{\rm G}_2(q)$, with $q=3^{2m+1}$ and $m\geqs 1$. 

\vs

This is very similar to the previous case. Here we take two cyclic maximal tori $T_i=\la x_i\ra$, $i=1,2$, of order $\Phi_{6}' = q+\sqrt{3q}+1$ and $\Phi_{6}''= q-\sqrt{3q}+1$, respectively, such that $|\Normalizer_{G}(T_i)/T_i|=6$, $(|x_1|,|x_2|)=1$ and $\M(x_i)=\{\Normalizer_{G}(T_i)\}$. Note that $\Phi_{6}'\Phi_{6}'' = q^2-q+1$. By arguing as in Case 1, we deduce that $|x_i|$ is a power of $r$ and $H=\Normalizer_{G}(T_j)$ for some distinct $i,j$. Let $x \in G$ be an element of order $9$ (see part (2) in the main theorem of \cite{Ward}, for example). Since $|H|$ is indivisible by $9$, it follows that $x$ is a derangement, but this is a contradiction since $r \neq 3$. 

\vs

\noindent \emph{Case 3.} $G = {}^2{\rm F}_4(q)$, with $q=2^{2m+1}$ and $m \geqs 1$.

\vs

Again, we proceed as in Case 1. Here $G$ has two cyclic maximal tori $T_i=\la x_i\ra$, $i=1,2$, where
\begin{align*}
|T_1| & = \Phi_{12}' = q^2 + \sqrt{2q^3}+q+\sqrt{2q}+1 \\
|T_2| & = \Phi_{12}'' = q^2 - \sqrt{2q^3}+q-\sqrt{2q}+1
\end{align*}
and $|\Normalizer_{G}(T_i)/T_i|=12$, $(|x_1|,|x_2|)=1$ and $\M(x_i)=\{\Normalizer_{G}(T_i)\}$. Note that $\Phi_{12}'\Phi_{12}'' = q^4-q^2+1$. As in Case 1, we see that $|x_i|$ is a power of $r$ and $H=\Normalizer_{G}(T_j)$ for some distinct $i,j$. Let $x \in G$ be an element of order $\ell_4(q)$. Since $|H|$ is indivisible by $\ell_4(q)$, it follows that $x \in \Delta(G)$, but this is a contradiction since $r \neq \ell_4(q)$.

\vs

\noindent \emph{Case 4.} $G = {\rm E}_8(q)$.

\vs

Again, we can proceed as in the previous cases, working with cyclic maximal tori $T_1,T_2$ and an element $x \in G$ of order $\ell_{24}(q)$, where
\begin{align*}
|T_1| & =\Phi_{15} = q^8-q^7+q^5-q^4+q^3-q+1 \\
|T_2| & =\Phi_{30} = q^8+q^7-q^5-q^4-q^3+q+1
\end{align*}
and $|\Normalizer_{G}(T_i)/T_i|=30$, $i=1,2$. We omit the details (note that $\ell_{24}(q)\in \pi(G)\setminus \pi(\Normalizer_{G}(T_i))$).

\vs

\noindent \emph{Case 5.} $G = {}^3{\rm D}_4(q)$.

\vs
 
As indicated in \cite[Table 6]{GM1}, $G$ has a maximal torus $T=\la x\ra$ of order $\Phi_{12} = q^4-q^2+1$ such that $|\Normalizer_{G}(T)/T|=4$ and $\M(x)=\{\Normalizer_{G}(T)\}$. 

Suppose that $x \not\in \Delta(G)$. Then $x^{G}\cap H$ is non-empty, and without loss of generality we may assume that $x \in H$ and thus $H=\Normalizer_{G}(T)$. If $q=2$ then 
$|H|=52$ and $|\pi(G)\setminus\pi(H)|=2$, so we must have $q>2$. Let $y_i \in G$ ($i=1,2$) be elements of order $\ell_i:=\ell_{m_i}(q)\geqs 5$, where $m_1=3$ and $m_2=6$. Since $|H|$ is indivisible by $\ell_1$ and $\ell_2$, it follows that $y_1,y_2 \in \Delta(G)$. But this is a contradiction since $\ell_1,\ell_2$ are distinct primes.  

Now assume that $x\in\Delta(G)$, so $|x|=\Phi_{12}$ is a power of $r$. If $q=2$ then $r=13$ and $H$ must contain elements of order $7,8,9,14,18,21$ and $28$, but no maximal subgroup of $G$ has this property (see \cite{ATLAS}, for example). Therefore, $q>2$. 
Following \cite[p.698]{GM2}, let $y \in G$ be an element of order $\Phi_3$ such that $|\Centralizer_{G}(y)|$ divides $\Phi_3^2$ and 
\[\mathcal{M}(y)=\{G_2(q),{\rm PGL}_3(q),(\Phi_6\circ \SL_3(q)).2d,\Phi_3^2.\SL_2(3)\},\]
where $d = (3,\Phi_3)$. Now $(\Phi_{12},\Phi_3) = 1$, so $y \not\in \Delta(G)$ and thus we may assume that $H \in \mathcal{M}(y)$. Let $z \in G$ be an element of order $\Phi_1\Phi_2\Phi_6 = (q^2-1)(q^2-q+1)$. Then $|H|$ is indivisible by $|z|$, so $z \in \Delta(G)$. But this is a contradiction since $(\Phi_{12}, \Phi_1\Phi_2\Phi_6) = 1$.

\vs

\noindent \emph{Case 6.} $G = {}^2{\rm E}_6(q)$.

\vs

Let $d=(3,q+1)$. As indicated in \cite[Table 6]{GM1} and \cite[Table 1]{GM2}, $G$ has two cyclic maximal tori  $T_i=\la x_i\ra$, $i=1,2$, of order $\Phi_{18}/d$ and $\Phi_6\Phi_{12}/d$, respectively. Then $(|x_1|,|x_2|)=1$ and  
$$\M(x_1)=\{ \SU_3(q^3).3 \},\;\; \M(x_2)=\left\{\begin{array}{ll}
\{\Phi_6.{}^3{\rm D}_4(q).3/d\} & \mbox{if $q>2$} \\
\{\Phi_6.{}^3{\rm D}_4(2), {\rm F}_4(2), {\rm Fi}_{22}\} & \mbox{if $q=2$.}
\end{array}\right.$$
No maximal subgroup of $G$ contains both $x_1$ and $x_2$ (see \cite[Table 10.5]{LPS}), so $x_i \in \Delta(G)$ for some $i$, and thus $|x_i|$ is a power of $r$.

First assume that $q=2$, so $|x_1|=19$, $|x_2|=13$ and thus $r \in \{13,19\}$.  If $r=13$, then $H$ contains a conjugate of $x_1$, so $H = \SU_3(8).3$ is the only option, but this is not possible since $|\pi(G) \setminus \pi(H)| = 4$. Similarly, if $r=19$ then $H \in \M(x_2)$ must contain elements of order $11,13$ and $17$, but it is easy to check that this is not the case.

Now assume that $q>2$. Let $x \in G$ be an element of order $\ell_{10}(q)$. Both $|\SU_3(q^3).3|$ and $|\Phi_6.{}^3{\rm D}_4(q).3/d|$ are indivisible by $\ell_{10}(q)$, so $x \in \Delta(G)$. However, this is not possible since $\ell_{10}(q)$ and $|x_i|$ are coprime.

\vs

\noindent \emph{Case 7.} $G = {\rm G}_2(q)$, $q \geqs 3$. 

\vs

We can use \textsf{GAP} \cite{GAP} to rule out the cases $q \leqs 5$, so we may assume that $q \geqs 7$.

First assume that $q=7$. By inspecting \cite[Table 6]{GM1} and \cite[Table 1]{GM2}, we see that $G$ has two cyclic maximal tori $T_i=\la x_i\ra$, $i=1,2$, of order $\Phi_6 = 43$ and $\Phi_3 = 57$, respectively, with $\M(x_1)=\{\SU_3(7).2\}$ and $\M(x_2)=\{\SL_3(7).2\}$. 
From \cite[Table 10.5]{LPS}, it follows that $x_i\in\Delta(G)$ for some $i$, so $H$ contains a conjugate of $x_j$, where $j=3-i$. Therefore, $H = \SL_3^\ep(7).2$ for some $\ep=\pm$. As noted in \cite[Table A.7]{KS}, $G$ contains elements of order $7^2+7=56$ and $7^2+7+1=57$. Now $\SU_3(7).2$ contains no element of order $57$, and $\SL_3(7).2$ has no element of order $56$. Therefore, $G$ always contains a derangement of non-prime power order, which is a contradiction.

For the remainder, we may assume that $q> 7$. We use the set-up in \cite[Section 5.7]{FMP}. Choose a $4$-tuple $(k_1,k_2,k_3,k_6)$ such that $(k_1,k_2)=1$, $k_i$ divides $\Phi_i$ for $i \in \{1,2\}$, $k_3=\Phi_3/(3,\Phi_3)$ and $k_6=\Phi_6/(3,\Phi_6)$. Note that the numbers  $k_1,k_2,k_3$ and $k_6$ are pairwise coprime. Let $y_1\in G$ be an element of order $k_6$, and fix a regular semisimple element $y_2\in G$ of order $k_1$. Similarly, fix $z_i\in G$, $i=1,2$, where $|z_1|=k_3$ and $z_2$ is a regular semisimple element of order $k_2$. 

From \cite[Table 10.5]{LPS}, it follows that either $y_1$ or $z_1$ is a derangement. Suppose that $y_1\in\Delta(G)$. Then $H$ contains a conjugate of $z_1$, so \cite[Lemma 5.27]{FMP} implies that $H = \SL_3(q).2$ is the only possibility. If $H$ also contains a conjugate of $z_2$, then $H=G$ by \cite[Corollary 5.28]{FMP}, a contradiction. Therefore $z_2\in\Delta(G)$, but once again we reach a contradiction since $(k_2,k_6)=1$. An entirely similar argument applies if $z_1\in\Delta(G)$.

\vs

\noindent \emph{Case 8.} $G \in \{{\rm E}_{6}(q), {\rm E}_7(q)\}$. 

\vs

First assume that $G = {\rm E}_{7}(q)$. Let $d=(2,q-1)$. As in \cite[Section 5.2]{FMP}, let $y_1,y_2\in G$ be elements of order 
$\Phi_{18}$ and $\Phi_2\Phi_{14}/d = (q^7+1)/d$, respectively, and let $z_1,z_2\in G$ be elements of order $\Phi_9$ and $\Phi_1\Phi_7/d = (q^7-1)/d$, respectively. From \cite[Corollary 5.6]{FMP}, we deduce that $y_i,z_j\in\Delta(G)$ for some $i,j \in \{1,2\}$. However, it is easy to check that $(|y_i|,|z_j|) = 1$ for all $i,j$, so this is a contradiction.

The case $G = {\rm E}_6(q)$ is entirely similar, using \cite[Corollary 5.11]{FMP} and elements $y_i,z_i \in G$ with $|y_1| = \Phi_9/d$, $|y_2| = \Phi_4$, $|z_1| = \Phi_3\Phi_{12}$ and $|z_2| = \Phi_5$ (where $d = (3,q-1)$).

\vs

\noindent \emph{Case 9.} $G = {\rm F}_4(q)$. 

\vs

For $q>2$, we can proceed as in Case 8, using the information in \cite[Section 5.5]{FMP}. The reader can check the details. 

Now assume that $q=2$. By inspecting \cite[Table 6]{GM1} and \cite[Table 1]{GM2}, we see that $G$ has two cyclic maximal tori $T_i=\la x_i\ra$, $i=1,2$, of order $\Phi_{12} = 13$ and 
$\Phi_8=17$, respectively, such that $\M(x_1)=\{{}^3{\rm D}_4(2).3,{}^2{\rm F}_4(2), {\rm L}_4(3).2_2\}$ and $\M(x_2)=\{{\rm Sp}_8(2)\}$. Therefore, $r \in \{13,17\}$. If $r=13$, then $H$ contains a conjugate of $x_2$, so $H = {\rm Sp}_{8}(2)$. However, \cite{ATLAS} indicates that $G$ has an element of order $28$, but ${\rm Sp}_8(2)$ does not, so this case is ruled out.   Therefore, $r=17$ and $H$ contains a conjugate of $x_1$, so $H\in \M(x_1)$. However, in each case one can check that $H$ does not contain an element of order $30$, but $G$ does. This final contradiction eliminates the case $G = {\rm F}_{4}(q)$.

\vs

This completes the proof of Proposition \ref{p:ex}.
\end{proof}

\subsection{Classical groups}\label{ss:class}

In order to complete the proof of Theorem \ref{t:main2}, we may assume that $G_0$ is a classical group over $\mathbb{F}_{q}$. Due to the existence of certain exceptional isomorphisms involving low-dimensional classical groups (see \cite[Proposition 2.9.1]{KL}, for example), and in view of our earlier work in Sections \ref{ss:sporadic}, \ref{ss:alt} and  \ref{ss:ex}, we may assume that $G_0$ is one of the groups listed in Table \ref{tab:gps}.

\renewcommand{\arraystretch}{1.1}
\begin{table}[h]
$$\begin{array}{ll} \hline
G_0 & \mbox{Conditions} \\ \hline
{\rm L}_{n}(q) & n \geqs 2,\, (n,q) \neq (2,2), (2,3), (2,4), (2,5), (2,9), (3,2), (4,2) \\
{\rm U}_{n}(q) & n \geqs 3,\, (n,q) \neq (3,2) \\
{\rm PSp}_{n}(q) & \mbox{$n \geqs 4$ even, $(n,q) \neq (4,2), (4,3)$} \\
{\rm P\Omega}_{n}^{\e}(q) & n \geqs 7 \\ \hline
\end{array}$$
\caption{Finite simple classical groups}\label{tab:gps} 
\end{table}
\renewcommand{\arraystretch}{1}

We will focus initially on the low-dimensional classical groups with socle ${\rm L}_{2}(q)$ and ${\rm L}_{3}^{\e}(q)$, which require special attention. As before, if a primitive prime divisor of $q^e-1$ exists, then $\ell_e(q)$ denotes the largest such prime divisor (as noted in Section \ref{s:prel}, if $e \geqs 2$, then $\ell_e(q)$ exists unless $(q,e) = (2,6)$, or $e=2$ and $q$ is a Mersenne prime).

\begin{lemma}\label{lem:L2even} 
Theorem \ref{t:main2} holds if $G = {\rm L}_2(q)$ and $q$ is even. 
\end{lemma}

\begin{proof} 
Write $q=2^f$, where $f \geqs 3$ (since ${\rm L}_{2}(4) \cong \Alt_5$, we may assume that $f \geqs 3$). The maximal subgroups of $G$ were originally classified by Dickson \cite{Dickson} (also see \cite[Tables 8.1 and 8.2]{BHR}); the possibilities for $H$ are as follows:
\begin{itemize}\addtolength{\itemsep}{0.2\baselineskip}
\item[{\rm (a)}] $H = (\mathbb{Z}_{2})^f{:}\mathbb{Z}_{q-1} = {\rm P}_1$ is a maximal parabolic subgroup of $G$;
\item[{\rm (b)}] $H = {\rm D}_{2(q\pm 1)}$;
\item[{\rm (c)}] $H = {\rm L}_{2}(q_0)$ with $q=q_0^e$, where $e$ is a prime and $q_0 \neq 2$.
\end{itemize}
The case $f=3$ can be handled using \textsf{GAP} \cite{GAP}, and we find that \eqref{e:star} holds if and only if $(H,r,\Ed(G))$ is one of the following (recall that $\Ed(G)$ denotes the set of orders of derangements in $G$):
$$({\rm P}_1, 3,\{3,9\}),\; ({\rm D}_{18}, 7, \{7\}),\; ({\rm D}_{14}, 3, \{3,9\}).$$
For the remainder, we may assume that $f \geqs 4$. 

Note that a Sylow $2$-subgroup of $G$ is self-centralizing and elementary abelian. In particular, if $x \in G$ then either $|x|=2$, or $|x|$ divides $q \pm 1$. Also note that $G$ contains elements of order $q\pm 1$, and it has a unique class of involutions. 

\vs

\noindent \emph{Case 1.} $H = {\rm P}_1$.

\vs

We claim that \eqref{e:star} holds if and only if $r=q+1$ is a Fermat prime. To see this, first observe that $|G:H|=q+1$ and $|H| = q(q-1)$ are relatively prime, so any element $x \in G$ of order $q+1$ is a derangement. Therefore, if \eqref{e:star} holds then $q+1=r^e$ for some $e\geqs 1$, and thus Lemma \ref{lem:primepower} implies that $f$ is a $2$-power and $e=1$ (so $r=q+1$ is a Fermat prime). 

For the converse, suppose that $q+1$ is a Fermat prime. We need to show that every derangement in $G$ has order $r=q+1$. Let $y \in \Delta(G)$, so $|y|$ divides $2$ or $q\pm 1$. But $q+1=r$ is a prime, so either $|y| \in \{2,r\}$ or $|y|$ divides $q-1$. 
Every involution has fixed points since $G$ has a unique class of involutions, so $|y|>2$. If $|y|$ divides $q-1$, then $y$ belongs to a maximal torus that is $G$-conjugate to the subgroup $\mathbb{Z}_{q-1}<H$. Again, this implies that $y$ has fixed points. Therefore, $|y|=r$ is the only possibility and the result follows. 

\vs

\noindent \emph{Case 2.} $H = {\rm D}_{2(q\pm 1)}$.

\vs

The case $H = {\rm D}_{2(q-1)}$ is identical to the previous one, and the same conclusion holds. A very similar argument also applies if $H = {\rm D}_{2(q+1)}$. Here any element of order $q-1$ is a derangement and by applying Lemma \ref{lem:primepower} we deduce that \eqref{e:star} holds if and only if $r=q-1$ is a Mersenne prime. 

\vs

\noindent \emph{Case 3.} $H = {\rm L}_{2}(q_0)$, where $q=q_0^e$, $e$ prime, $q_0 \neq 2$.

\vs

Finally, observe that subfield subgroups are easily eliminated since elements of order $q \pm 1$ are derangements.
\end{proof}

\begin{lemma}\label{lem:L2ext}
Theorem \ref{t:main2} holds if $G_0 = {\rm L}_2(q)$ and $q$ is even.
\end{lemma}

\begin{proof} 
As before, write $q=2^f$, where $f \geqs 3$. In view of Lemma \ref{lem:L2even}, we may assume that 
$$G = G_0. \la \phi \ra \leqs {\rm \Gamma L}_{2}(q) = {\rm Aut}(G_0),$$ 
where $\phi$ is a nontrivial field automorphism of $G_0$, so the order of $\phi$ divides $f$. The case $f=3$ can be handled directly, using \cite{GAP} for example. Here $G = {\rm \Gamma L}_{2}(8)$ and we find that \eqref{e:star} holds if and only if $(H,r,\Ed(G))$ is one of the following:
$$(\Normalizer_{G}({\rm P}_1), 3,\{3,9\}),\; (\Normalizer_{G}({\rm D}_{18}), 7, \{7\}),\; (\Normalizer_{G}({\rm D}_{14}), 3, \{3,9\}).$$
For the remainder, we may assume that $f \geqs 4$. 

Since $G_0 \not\leqs H$, we have $G=G_0H$. Set $H_0 = H \cap G_0$ and note that $H_0$ is a maximal subgroup of $G_0$ (see \cite[Table 8.1]{BHR}). As in \eqref{e:g0}, we have $\Delta_{H_0}(G_0) \subseteq \Delta_H(G)$, whence Lemma \ref{lem:L2even} implies that \eqref{e:star} holds only if one of the following holds:
\begin{itemize}\addtolength{\itemsep}{0.2\baselineskip}
\item[{\rm (a)}] $H_0={\rm P}_1$, $r=q+1$ is a Fermat prime;
\item[{\rm (b)}] $H_0={\rm D}_{2(q+1)}$, $r=q-1$ is a Mersenne prime;
\item[{\rm (c)}] $H_0={\rm D}_{2(q-1)}$, $r=q+1$ is a Fermat prime.
\end{itemize}
We consider each of these cases in turn.

\vs

\noindent \emph{Case 1.} $H_0={\rm D}_{2(q+1)}$, $r=q-1$ is a Mersenne prime.

\vs

Here $f \geqs 5$ is a prime, so $G = {\rm \Gamma L}_{2}(q) = G_0. \la \phi \ra$ and $H=H_0.\la \phi \ra$ is the only possibility, where $\phi$ has order $f$. Note that 
$$\Centralizer_G(\phi) = {\rm L}_{2}(2) \times \la \phi \ra \cong \SSS_3 \times \mathbb{Z}_{f},$$
so if $x\in G$ then either $|x|\in \{2,r,f,2f,3f\}$, or $|x|$ divides $q+1$. We claim that $\Ed(G)=\{r\}$. Note that $\la \phi \ra$ is a Sylow $f$-subgroup of $G$.

Let $y \in G$ be a nontrivial element. If $|y| \in \{2,f\}$, or if $|y|$ divides $q+1$, then $y$ is conjugate to an element of $H$ and thus $y$ has fixed points. Next suppose that $|y|=kf$ and $k \in \{2,3\}$. 
Then $|y^k|=f$ and thus $y^k$ is $G$-conjugate to $\phi^i$ for some $1 \leqs i<f$. Without loss of generality, we may assume that $y^k=\phi$, so $y \in\Centralizer_G(\phi)$. Since $|H|=2(q+1)f$, $H$ has a Sylow $2$-group $R=\la u\ra \cong \mathbb{Z}_{2}$ and a normal $2$-complement $V\la \phi \ra$ of order $(q+1)f$, where $V \cong \mathbb{Z}_{q+1}$. Since $\phi$ normalizes $H_0=VR$, we deduce that $\phi$ centralizes $R$. Now $q+1$ is divisible by $3$, so $V$ has a unique subgroup of order $3$, say $\la x\ra$. Then the involution $u$ inverts $x$, and $\phi$ centralizes $x$ since $|\phi|=f \geqs 5$ is odd. Thus $\SSS_3\cong \la u,x \ra\leqs \Centralizer_{G_0}(\phi)$, which implies that $\Centralizer_G(\phi)=\la u,x \ra \times \la \phi \ra\leqs H$. Therefore, $y \in H$. We conclude that every derangement in $G$ has order $r$, as required. 

\vs

In the two remaining cases, $r=q+1$ is a Fermat prime and $f=2^m$ for some integer $m \geqs 2$. In both cases, we claim that \eqref{e:star} does not hold. 
In order to see this, we may assume that the index of $G_0$ in $G$ is a prime number, which in this case implies that $|G:G_0|=2$, so $G = G_0.\la \phi \ra$ and $\phi$ is an involutory field automorphism of $G_0$. Indeed, if $G_0 \normeq G_1\normeq G$ then $G=HG_1$ and Lemma \ref{lem:normal} implies that $\Delta_L(G_1)\subseteq\Delta(G)$ for any subgroup $L$ of $G_1$ containing $G_1\cap H$.

\vs

\noindent \emph{Case 2.} $H_0={\rm D}_{2(q-1)}$, $r=q+1$ is a Fermat prime.

\vs

By the above comments, we may assume that $G = G_0.\la \phi \ra$ and $H={\rm D}_{2(q-1)}.\la \phi \ra$, where $\phi$ has order $2$. Note that $\Centralizer_G(\phi)={\rm L}_2(2^{f/2}) \times \la \phi\ra$. 
Since $\Centralizer_G(\phi)$ does not contain a Sylow $2$-subgroup of $G$, we deduce that the Sylow $2$-subgroups of $G$ are nonabelian. Therefore $G$ contains an element $z$ of order $4$. However, the Sylow $2$-subgroups of $H$ are isomorphic to $C_2 \times C_2$, so $z\in\Delta(G)$. We conclude that $G$ contains derangements of order $r$ and $4$, so \eqref{e:star} does not hold. 

\vs

\noindent \emph{Case 3.} $H_0={\rm P}_1$, $r=q+1$ is a Fermat prime.

\vs

Finally, let us assume that $H = \Normalizer_G({\rm P}_1) = {\rm P}_{1}.\la \phi \ra = H_0.\la \phi \ra$, where $|\phi|=2$. As above, we have $\Centralizer_G(\phi)={\rm L}_2(2^{f/2})\times \la \phi\ra$, so 
$\Centralizer_G(\phi)$ contains an element  of order $2(q_0+1)$, where $q_0 = 2^{f/2}$. We claim that $H$ does not contain such an element. Seeking a contradiction, suppose $x \in H$ has order $2(q_0+1)$. Since $H=H_0\cup H_0\phi$ and $H_0={\rm P}_1$ has no element of order $2(q_0+1)$, we deduce that $x\in H_0\phi$ and we may write $x=u\phi$ with $u\in H_0$. 
In terms of matrices (and a suitable basis for the natural ${\rm L}_{2}(q)$-module), we have 
$$u=\left(\begin{array}{cc}\lambda  & a \\  0 & \lambda^{-1} \end{array}\right)$$ 
where $\lambda,a\in\F_{q}$ and $\lambda\neq 0$.
Then $x^2=(u\phi)(u\phi)=uu^\phi$ has order $q_0+1$. We may assume that $\phi$ is the standard field automorphism of order $2$ with respect to this basis, so
\[x^2=uu^\phi=\left(\begin{array}{cc}\lambda  & a \\  0 & \lambda^{-1} \end{array}\right)\left(\begin{array}{cc}\lambda^{q_0}  & a^{q_0} \\  0 & \lambda^{-q_0} \end{array}\right)=\left(\begin{array}{cc}\lambda^{1+q_0}  & b \\  0 & \lambda^{-1-q_0} \end{array}\right)\] 
with $b=\lambda^{q_0}+a\lambda^{-q_0}$. Since $x^2$ has order $q_0+1$ we deduce that 
$\lambda^{2(q_0+1)}=1$, which implies that $\lambda^{q_0+1}=1$ since $\F_q$ has characteristic $2$. Therefore
$$x^2=\left(\begin{array}{cc}1  & b \\  0 & 1 \end{array}\right)$$ 
has order $q_0+1$, which is absurd. This justifies the claim, and we deduce that $\Delta(G)$ contains elements of order $2(q_0+1)$. In particular, \eqref{e:star} does not hold.
\end{proof}

\begin{lemma}\label{lem:L2odd}
Theorem \ref{t:main2} holds if $G_0 = {\rm L}_2(q)$ and $q$ is odd.
\end{lemma}

\begin{proof}
Write $q=p^f$, where $p$ is an odd prime. In view of the isomorphisms ${\rm L}_{2}(5) \cong \Alt_5$ and ${\rm L}_{2}(9) \cong \Alt_6$, we may assume that $q \geqs 7$ and $q \neq 9$. The case $q=7$ can be checked directly using \textsf{GAP}, and we find that \eqref{e:star} holds if and only if $(G,H,r,\Ed(G))$ is one of the following:
$$({\rm L}_{2}(7), {\rm P}_{1}, 2,\{2,4\}),\; ({\rm L}_{2}(7), \SSS_4, 7,\{7\}),\; ({\rm PGL}_{2}(7), \Normalizer_{G}({\rm P}_{1}), 2,\{2,4,8\}).$$
For the remainder, we may assume that $q \geqs 11$.

\vs

\noindent \emph{Case 1.} $G=G_0$.

\vs

First assume that $G = {\rm L}_{2}(q)$. The maximal subgroups of $G$ are well known (see \cite[Tables 8.1 and 8.2]{BHR}); the possibilities for $H$ are as follows:
\begin{itemize}\addtolength{\itemsep}{0.2\baselineskip}
\item[{\rm (a)}] $H = (\mathbb{Z}_{p})^f{:}\mathbb{Z}_{(q-1)/2} = {\rm P}_1$ is a maximal parabolic subgroup of $G$; 
\item[{\rm (b)}] $H = {\rm D}_{q-\e}$, where $q \geqs 13$ if $\e=1$;
\item[{\rm (c)}] $H = {\rm L}_{2}(q_0)$, where $q=q_0^e$ for some odd prime $e$;
\item[{\rm (d)}] $H = {\rm PGL}_{2}(q_0)$, where $q=q_0^2$;
\item[{\rm (e)}] $H=\Alt_5$, where $q \equiv \pm 1 \imod{10}$ and either $q=p$, or $q=p^2$ and $p \equiv  \pm 3 \imod{10}$;
\item[{\rm (f)}] $H=\Alt_4$, where $q =p \equiv \pm 3 \imod{8}$ and $q \not\equiv \pm 1 \imod{10}$;
\item[{\rm (g)}] $H=\SSS_4$, where $q=p \equiv \pm 1 \imod{8}$.
\end{itemize}
Note that $G$ contains elements of order $(q \pm 1)/2$, and a unique conjugacy class of involutions.

If $H$ is a subfield subgroup (as in (c) or (d) above), then it is clear that any element of order $(q \pm 1)/2$ is a derangement, so property \eqref{e:star} does not hold in this situation. Similarly, it is straightforward to handle the cases $H \in \{\Alt_5, \Alt_4, \SSS_4\}$. For example, suppose $H = \Alt_5$, so $q \equiv \pm 1 \imod{10}$ and either $q=p$, or $q=p^2$ and $p \equiv  \pm 3 \imod{10}$. Note that every nontrivial element of $H$ has order $2,3$ or $5$. If $q \geqs 19$ then any element of order $(q \pm 1)/2$ is a derangement; if $q=11$, then elements of order $6$ are derangements. The cases $H = \Alt_4$ and $\SSS_4$ are just as easy.

If $H = {\rm D}_{q - 1}$ then any element in $G$ of order $p$ or $(q+1)/2$ is a derangement, and the dihedral groups of order $q+1$ can be eliminated in a similar fashion.

Finally, let us assume that $H = {\rm P}_{1} = (\mathbb{Z}_{p})^f{:}\mathbb{Z}_{(q-1)/2}$, so $|H| = q(q-1)/2$ and $|G:H|=q+1$. We claim that \eqref{e:star} holds if and only if $q=2r^e-1$ for some positive integer $e$.

First observe that any element of order $(q+1)/2$ is a derangement, so if \eqref{e:star} holds then $q=2r^e-1$ for some $e \in \mathbb{N}$. For the converse, suppose that $q=2r^e-1$. We claim that 
$$\Ed(G) = \{r^i \,:\, 1 \leqs i \leqs e\}.$$ 
Since $|H|$ is indivisible by $r$, the inclusion $\{r^i \,:\, 1 \leqs i \leqs e\} \subseteq \Ed(G)$ is clear. To see that equality holds, 
let $y \in G$ be a nontrivial element, and suppose that $|y|$ is divisible by a prime $s \neq r$. Since a Sylow $p$-subgroup of $G$ is self-centralizing, it follows that either $|y| = p$, or $|y|=2$ and $r$ is odd, or $|y|$ is a divisor of $(q-1)/2$. In the first two cases, it is clear that $y$ has fixed points, so let us assume that $|y|$ divides $(q-1)/2$. Then $y$ is conjugate to an element of the maximal torus $\mathbb{Z}_{(q-1)/2} < H$, so once again $y$ has fixed points. This justifies the claim.

\vs

\noindent \emph{Case 2.} $G \neq G_0$.

\vs

To complete the proof of the lemma, we may assume that $G \neq G_0$, $q \geqs 11$ and $H_0 = H \cap G_0 = {\rm P}_{1}$, in which case \eqref{e:star} holds only if $q=2r^e-1$ for some positive integer $e$ (note that $H \cap G_0$ is a maximal subgroup of $G_0$). There are several possibilities for $G$.

First assume that $G = {\rm PGL}_{2}(q)$, so $H = (\mathbb{Z}_{p})^f{:}\Z_{q-1}$. We claim that \eqref{e:star} holds if and only if $r=2$ and $q=2^{e+1}-1$ is a Mersenne prime. As above, any element of order $(q+1)/2$ is a derangement. Now $G$ also contains elements of order $q+1$, and they are also derangements. Therefore, if \eqref{e:star} holds then $r=2$ is the only possibility, so $p^f+1 = 2^{e+1}$ and Lemma \ref{lem:primepower} implies that  $q=p=2^{e+1}-1$ is a Mersenne prime. 

For the converse, suppose that $q=p=2^{e+1}-1$ is a Mersenne prime. We claim that 
$$\Ed(G) = \{2^i \,:\, 1 \leqs i \leqs e+1\}.$$ 
As above, any involution in $G_0$ is a derangement, and so is any element in $G$ of order $2^i$ with $1< i \leqs e+1$ since $|H|_2 = 2$, hence $\{2^i \,:\, 1 \leqs i \leqs e+1\} \subseteq \Ed(G)$. To see that equality holds, suppose that $y \in G$ has order divisible by an odd prime. Then either $|y|=p$, or $y$ is conjugate to an element of the maximal torus $\mathbb{Z}_{q-1} < H$; in both cases, $y$ has fixed points. The result follows. 

To complete the proof of the lemma, we may assume that $G = G_0.\la \phi \ra$ or $G_0.\la \delta\phi \ra$, where $\phi$ is a nontrivial field automorphism of $G_0$ of order $e$ (so $e$ divides $f$) and $\delta = {\rm diag}(\omega_1,\omega_2) \in {\rm PGL}_{2}(q)$ (modulo scalars) is a diagonal automorphism of $G_0$. Recall that $(q+1)/2 = r^e$ for some prime $r$ and positive integer $e$. Our goal is to show that \eqref{e:star} does not hold. 

First observe that $r$ is odd. Indeed, if $r=2$ then $p^f+1 = 2^{e+1}$ and thus Lemma \ref{lem:primepower} implies that $f=1$, which is false. Next we claim that $f$ is a $2$-power. 
To see this, first assume that $f$ is odd and $p=2^t-1$ is a Mersenne prime. Then $r^e=(p^f+1)/2$ is divisible by $(p+1)/2=2^{t-1}$, but $r$ is odd so this is not possible. For the general case, suppose that $f=2^am$ where $a \geqs 0$ and $m>1$ is odd (and we may assume that $a>0$ if $p$ is a Mersenne prime). We now proceed as in the proof of Lemma \ref{l:nagell}(iii). We have
$$r^e = \frac{q^2-1}{2(q-1)} = \frac{p^{2^{a+1}m}-1}{2(p^{2^{a}m}-1)}$$
and thus $r = \ell_{2f}(p)$. Set $s=\ell_{2^{a+1}}(p)$ (note that $s$ exists since $a>0$ if $p$ is a Mersenne prime). Now $f = 2^am$ is indivisible by $2^{a+1}$, so $(s,q-1)=1$ and thus $s$ does not divide $2(q-1)$. Therefore, $r=s$ is the only possibility, but this is a contradiction since $2f=2^{a+1}m>2^{a+1}$. This justifies the claim. 

Therefore, in order to show that \eqref{e:star} does not hold, we may assume that $|G:G_0|=2$. Write $G = G_0 \cup G_0\gamma$.

If we identify $\Omega$ with the set of $1$-dimensional subspaces of the natural ${\rm L}_{2}(q)$-module, then $\phi$ and $\delta\phi$ fix the $1$-spaces $\la (1,0) \ra$ and $\la (0,1) \ra$. Therefore, Lemma \ref{l:gms} implies that the coset $G_0\gamma$ contains derangements. But every element in this coset has even order, which is incompatible with property \eqref{e:star}. 
\end{proof}

To summarize, we have now established the following result. (Note that the case appearing in the final row of Table \ref{tab:l2} is recorded as $(G,H) = ({\rm L}_{3}(2), {\rm P}_{1})$ in Table \ref{tab:main}.)

\begin{proposition}\label{p:l2}
Let $G$ be a finite almost simple primitive permutation group with point stabilizer $H$ and socle ${\rm L}_{2}(q)$, where $q \geqs 4$ and $q \neq 5$. Then \eqref{e:star} holds if and only if $(G,H,r)$ is one of the cases in Table \ref{tab:l2}.
\end{proposition}

\renewcommand{\arraystretch}{1.1}
\begin{table}[h]
$$\begin{array}{lllll} \hline
G & H & r & \Ed(G) & \mbox{Conditions} \\ \hline
{\rm \Gamma L}_2(q) & \Normalizer_{G}({\rm D}_{2(q+1)}) & r & r & \mbox{$r=q-1$ Mersenne prime} \\
{\rm \Gamma L}_{2}(8) & \Normalizer_{G}({\rm P}_1), \Normalizer_{G}({\rm D}_{14}) & 3 & 3,9 & \\ 
{\rm PGL}_{2}(q) & \Normalizer_{G}({\rm P}_{1}) & 2 & 2^i, \, 1 \leqs i \leqs e+1 & \mbox{$q=2^{e+1}-1$ Mersenne prime} \\
{\rm L}_2(q) & {\rm P}_1 & r & r^i, \, 1 \leqs i \leqs e & q=2r^e-1 \\
& {\rm P}_1, {\rm D}_{2(q-1)} & r & r & \mbox{$r=q+1$ Fermat prime} \\
& {\rm D}_{2(q+1)} & r & r & \mbox{$r=q-1$ Mersenne prime} \\
{\rm L}_{2}(9) & {\rm P}_{1} & 5 & 5 & \\
{\rm L}_{2}(8) & {\rm P}_{1}, {\rm D}_{14} & 3 & 3,9 & \\
{\rm L}_{2}(7) & \SSS_4 & 7 & 7 & \\ \hline
\end{array}$$
\caption{The cases $(G,H,r)$ in Proposition \ref{p:l2}}
\label{tab:l2}
\end{table}
\renewcommand{\arraystretch}{1}

\begin{lemma}\label{lem:L3}
Theorem \ref{t:main2} holds if $G_0 = {\rm L}_3(q)$.
\end{lemma}

\begin{proof}
Set $d=(3,q-1)$ and note that $G_0$ contains elements of order $(q^2+q+1)/d$ and $(q^2-1)/d$. We may assume that $q \geqs 3$ since ${\rm L}_{3}(2) \cong {\rm L}_{2}(7)$. If $3 \leqs q \leqs 7$, then we can use \textsf{GAP} to verify the desired result; we find that \eqref{e:star} holds if and only if $G = {\rm L}_{3}(q)$, $H \in \{{\rm P}_{1},{\rm P}_{2}\}$ and $r = (q^2+q+1)/d$, in which case $\Ed(G) = \{r\}$ (note that $(q^2+q+1)/d$ is a prime number for all $q \in \{3,4,5,7\}$). For the remainder, we will assume that $q \geqs 8$. In particular, note that $(q^2-1)/d$ is not a prime power (indeed, it is easy to check that $(q^2-1)/d$ is a prime-power if and only if $q=3$ or $7$).

\vs

\noindent \emph{Case 1.} $G=G_0$.

\vs

First assume that $G = {\rm L}_{3}(q)$. The possibilities for $H$ are given in \cite[Tables 8.3 and 8.4]{BHR}. We can immediately eliminate any subgroup $H$ that does not contain an element of order $(q^2-1)/d$, so this implies that $H$ is either a maximal parabolic subgroup, or $H = {\rm SO}_{3}(q)$ (with $q$ odd). 

Suppose that $H$ is a maximal parabolic subgroup. Without loss of generality, we may assume that $H  = {\rm P}_{1}$ (the actions of $G$ on $1$-spaces and $2$-spaces are permutation isomorphic), so $|H| = q^3(q-1)(q^2-1)/d$. We claim that $G$ has property \eqref{e:star} if and only if one of the following holds:
\renewcommand{\arraystretch}{1.2}
\begin{equation}\label{e:ab}
\begin{array}{ll}
\mbox{(a)} & \mbox{$d=1$ and $q^2+q+1 = r$; or} \\
\mbox{(b)} & \mbox{$d=3$ and $q^2+q+1  \in \{3r,3r^2\}$.}
\end{array}
\end{equation}

\renewcommand{\arraystretch}{1}

To see this, first notice that any element $x \in G$ of order $(q^2+q+1)/d$ is a derangement. Therefore, if \eqref{e:star} holds then $(q^2+q+1)/d = r^e$ for some positive integer $e$, and by applying Lemma \ref{l:nagell} we deduce that (a) or (b) holds.  Conversely, suppose that (a) or (b) holds. We claim that 
$$\Ed(G) = \left\{\begin{array}{ll}
\{r,r^2\} & \mbox{if $d=3$ and $q^2+q+1 = 3r^2$} \\
\{r\} & \mbox{otherwise.}
\end{array}\right.$$
To see this, we use the fact that the action of $G$ on $1$-spaces is doubly transitive, so the corresponding permutation character has the form $1_H^G = 1+\chi$ for some irreducible character $\chi \in {\rm Irr}(G)$ of degree $q(q+1)$. By inspecting the character table of $G$ (see \cite[Table 2]{FS}, for example), we see that $\chi(x)=-1$ if and only if $x$ has order $r$ (or $r^2$ if $d=3$ and $q^2+q+1 = 3r^2$). This justifies the claim.

Now assume that $H = {\rm SO}_{3}(q)$, so $q$ is odd. Here elements of order $(q^2+q+1)/d$ are derangements, and so is any unipotent element with Jordan form $[J_2,J_1]$ (where $J_i$ denotes a standard unipotent Jordan block of size $i$). Therefore, \eqref{e:star} does not hold in this situation.

\vs

\noindent \emph{Case 2.} $G \neq G_0$.

\vs

To complete the proof of the lemma, we may assume that $G \neq G_0$ and $q \geqs 8$. Let $M$ be a maximal subgroup of $G_0$ containing $H_0 := H \cap G_0$. From the analysis in Case 1, we may assume that $M = {\rm P}_{1}$, in which case $H_0$ is either equal to ${\rm P}_{1}$, or it is a non-maximal subgroup of type ${\rm P}_{1,2}$ (a Borel subgroup of $G_0$) or ${\rm GL}_{2}(q) \times {\rm GL}_{1}(q)$. We can quickly eliminate the latter two possibilities. For instance, if $H_0$ is a Borel subgroup then $\Delta_{H_0}(G_0)$ contains all elements of order $(q^2-1)/d$, so \eqref{e:star} does not hold (see \eqref{e:g0}). Similarly, if $H_0$ is of type  ${\rm GL}_{2}(q) \times {\rm GL}_{1}(q)$ then $\Delta_{H_0}(G_0)$ contains elements of order $(q^2+q+1)/d$, and also unipotent elements with Jordan form $[J_3]$. 

Therefore, we may assume that $H_0 = {\rm P}_{1}$, with $q \geqs 8$. To show that \eqref{e:star} does not hold, we may as well assume that we are in one of the two cases (a) and (b) in \eqref{e:ab} above (otherwise the conclusion is clear). Note that the condition $H_0 = {\rm P}_{1}$ implies that $G \leqs {\rm \Gamma L}_{3}(q)$ (that is, $G$ does not contain a graph automorphism). Also note that we may identify $\Omega$ with the set of $1$-dimensional subspaces of the natural ${\rm L}_{3}(q)$-module. Note that $r>3$.

First assume that $G = {\rm PGL}_{3}(q)$, so $d=3$ since we are assuming that $G \neq G_0$.
Here $G$ has a cyclic maximal torus $\la x \ra$ of order $q^2+q+1$. Then $x$ is a derangement and thus \eqref{e:star} does not hold since $q^2+q+1$ is not a prime power (note that $(q^2+q+1)_3=3$). 

For the remainder, we may assume that $q=p^f$ and $f \geqs 2$ (also recall that $q \geqs 8$). In view of \eqref{e:ab}, Lemma \ref{l:nagell}(iii) implies that $f$ is a $3$-power. To deduce that \eqref{e:star} does not hold, we may assume that $|G:G_0|$ is a prime number. Since $G \leqs {\rm \Gamma L}_{3}(q)$ and $f$ is a $3$-power, we may assume that $|G:G_0|=3$ and thus 
$G=G_0.\la \phi \ra$ or $G_0.\la \delta\phi\ra$, where $\phi$ is a field automorphism of order $3$ and $\delta$ is an appropriate diagonal automorphism ${\rm diag}(\omega_1, \omega_2,\omega_3) \in {\rm PGL}_{3}(q)$ (modulo scalars). In both cases, the result follows by applying Lemma \ref{l:gms}. For example, $\delta\phi$ has more than one fixed point on $\Omega$, so Lemma \ref{l:gms} implies that the coset $G_0\delta\phi$ contains derangements, none of which has $r$-power order. In view of this final contradiction, we conclude that \eqref{e:star} does not hold if $G \neq G_0$. 
\end{proof}

\begin{lemma}\label{lem:U3}
Theorem \ref{t:main2} holds if $G_0 = {\rm U}_3(q)$.
\end{lemma}

\begin{proof}
Set $d=(3,q+1)$ and observe that $G_0$ contains elements of order $(q^2-q+1)/d$ and $(q^2-1)/d$. Note that $(q^2-1)/d$ is a prime power if and only if $q \in \{3,5\}$. In order to show that \eqref{e:star} does not hold, we may assume that $G = G_0$.

The cases $q \in \{3,4,5\}$ can be handled directly, using \textsf{GAP}, so for the remainder we will assume that $q \geqs 7$. Let $V$ be the natural $G_0$-module, and let ${\rm P}_{1}$ (respectively ${\rm N}_{1}$) be the $G_0$-stabilizer of a $1$-dimensional totally isotropic (respectively, non-degenerate) subspace of $V$. Note that ${\rm N}_{1}$ is a subgroup of type ${\rm GU}_{1}(q) \times {\rm GU}_{2}(q)$.

We can immediately rule out any subgroup $H$ that does not contain elements of order $(q^2-1)/d$, which means that we may assume $H$ is of type ${\rm P}_{1}$, ${\rm N}_{1}$ or $O_{3}(q)$ ($q$ odd). In all three cases, elements of order $(q^2-q+1)/d$ are derangements. In addition, if $H = {\rm N}_{1}$ (respectively, ${\rm SO}_{3}(q)$) then unipotent elements with Jordan form $[J_3]$ (respectively, $[J_2,J_1]$) are derangements. Finally, suppose that $H = {\rm P}_{1}$. Let $\omega \in \mathbb{F}_{q^2}$ be an element of order $q+1$ and set $x = {\rm diag}(1,\omega,\omega^{-1}) \in G$ (modulo scalars) with respect to an orthonormal basis for $V$. Then $x$ does not fix a totally isotropic $1$-space, whence $x$ is a derangement of order $q+1$.
\end{proof}

Having handled the low-dimensional groups, we are now in a position to complete the proof of Theorem \ref{t:main2} for linear and unitary groups.

\begin{lemma}\label{lem:Ln}
Theorem \ref{t:main2} holds if $G_0 = {\rm L}_n^{\e}(q)$. 
\end{lemma}

\begin{proof}
We may assume that $n \geqs 4$. Set $d=(n,q-\e)$ and $e=(q-\e)d$. Let $V$ be the natural $G_0$-module. Let ${\rm P}_{i}$ be the $G_0$-stabilizer of a totally isotropic $i$-dimensional  subspace of $V$ (so ${\rm P}_{i}$ is a maximal parabolic subgroup of $G_0$, and we can take any $i$-space if $\e=+$). Similarly, if $\e=-$ then let ${\rm N}_{i}$ denote the $G_0$-stabilizer of an $i$-dimensional non-degenerate subspace of $V$ (so ${\rm N}_{i}$ is of type ${\rm GU}_{i}(q) \times {\rm GU}_{n-i}(q)$). In order to show that  \eqref{e:star} does not hold, we may assume that $G=G_0$. There are several cases to consider.

\vs

\noindent \emph{Case 1.} $n=2m$ and $m \geqs 4-\e$ is odd. 

\vs

First assume that $m \geqs 5$. As in the proof of \cite[Proposition 3.11]{BTV}, let $x \in G$ be an element of order $(q^{m+2}-\e)(q^{m-2}-\e)/e$. Then $|x|$ is not a prime power (see Lemma \ref{lem:numerical}), and \cite[Table II]{GK} indicates that $x$ is a derangement unless one of the following holds:
\begin{itemize}\addtolength{\itemsep}{0.2\baselineskip}
\item[{\rm (a)}] $\e=+$ and $H={\rm P}_{m-2}$ (or ${\rm P}_{m+2}$);
\item[{\rm (b)}] $\e=-$ and $H = {\rm N}_{m-2}$.
\end{itemize}
In (a), any element of order $\ell_{n}(q)$ or $\ell_{n-1}(q)$ is a derangement, and elements of order $\ell_{n}(q)$ and $\ell_{2(n-1)}(q)$ are derangements in case (b).  

Now assume $m=3$, so $(\e,n)=(+,6)$. 
Let $x \in G$ be an element of order $(q^6-1)/e$, which is not a prime power by Lemma \ref{l:ppower}. Here $x$ is a \emph{Singer element}, and the main theorem of \cite{Ber} implies that $x$ is a derangement, unless  
$H$ is a field extension subgroup, so we have reduced to the case where $H$ is of type ${\rm GL}_{3}(q^2)$ or ${\rm GL}_{2}(q^3)$. In this situation, elements of order $\ell_{5}(q)$ are derangements, and so are unipotent elements with Jordan form $[J_2,J_1^4]$. 

\vs

\noindent \emph{Case 2.} $n=2m$ and $m \geqs 3-\e$ is even. 

\vs

First assume that $m \geqs 4$. Let $x \in G$ be an element of order $(q^{m+1}-\e)(q^{m-1}-\e)/e$. Then Lemma \ref{lem:numerical} implies that $|x|$ is not a prime power, and from \cite[Table II]{GK} we deduce that $x$ is a derangement unless 
one of the following holds:
\begin{itemize}\addtolength{\itemsep}{0.2\baselineskip}
\item[{\rm (a)$'$}] $\e=+$ and $H={\rm P}_{m-1}$ (or ${\rm P}_{m+1}$);
\item[{\rm (b)$'$}] $\e=-$ and $H= {\rm N}_{m-1}$.
\end{itemize}
To deal with these cases, we can repeat the argument in Case 1. 

Now assume $m=2$, so $(\e,n)=(+,4)$. 
By applying the main theorem of \cite{Ber}, we deduce that elements of order $(q^4-1)/e$ are derangements unless $H$ is a field extension subgroup of type ${\rm GL}_{2}(q^2)$. Moreover, since $(q^4-1)/e$ is not a prime power (see Lemma \ref{l:ppower}), we can assume that $H$ is of type ${\rm GL}_{2}(q^2)$. Here elements of order $\ell_{3}(q)$ and unipotent elements with Jordan form $[J_2,J_1^2]$ are derangements. 

\vs

\noindent \emph{Case 3.} $\e=+$, $n=2m+1$ and $m \geqs 2$. 

\vs

If $G = {\rm L}_{11}(2)$, then any element of order $2^{11}-1 = 23 \cdot 89$ is a derangement, unless $H$ is a field extension subgroup of type ${\rm GL}_{1}(2^{11})$, in which case elements of order $2^{10}-1$ are derangements. For the remainder, we may assume that $(n,q) \neq (11,2)$.

Let $x \in G$ be an element of order $(q^{m+1}-1)(q^{m}-1)/e$. By Lemmas \ref{lem:numerical} and \ref{l:ppower}, $|x|$ is not a prime power, so we may assume that $H = {\rm P}_{m}$ (see \cite[Table II]{GK}). If $m \geqs 3$, then elements of order $\ell_{n}(q)$ or $\ell_{n-2}(q)$ are derangements. Similarly, if $m=2$ then we can take elements of order $\ell_{5}(q)$ or $\ell_{4}(q)$.

\vs

\noindent \emph{Case 4.} $\e=-$, $n=2m+1$ and $m \geqs 4$. 

\vs

Fix $x \in G$, where 
$$|x| = \left\{\begin{array}{ll}
(q^{m+1}+1)(q^{m}-1)/e & \mbox{$m$ even} \\
(q^{m+2}+1)(q^{m-1}-1)/e & \mbox{$m$ odd.} 
\end{array}\right.$$
By Lemma \ref{lem:numerical}, $|x|$ is not a prime power, and \cite[Table II]{GK} indicates that $x$ has fixed points only if $H$ stabilizes a subspace $U$ of $V$ with $\dim U \geqs 2$. Therefore, we may assume that $H$ has this property, in which case any element of order $\ell_{2n}(q)$ or $\ell_{n-1}(q)$ is a derangement. 

\vs

\noindent \emph{Case 5.} $\e=-$ and $n \in \{4,5,6,7\}$. 

\vs

First assume that $n=7$. Let $x \in G$ be an element of order $(q^6-1)/d$. Since $|x|$ is not a prime power, by inspecting the list of maximal subgroups of $G$ (see \cite[Tables 8.37 and 8.38]{BHR}) it follows that we can assume that $H \in \{{\rm P}_{3}, {\rm N}_{1}, {\rm SO}_{7}(q)\}$. In all three cases, any element of order $\ell_{14}(q)$ is a derangement. Similarly, elements of order $\ell_{10}(q)$ are derangements, unless $H = {\rm N}_{1}$, in which case any unipotent element with Jordan form $[J_7]$ is a derangement. The case $n=5$ is entirely similar.

Next assume that $n=6$. For now, let us assume that $q \not\in \{2,5\}$. Let $x \in G$ be an element of order $(q^5+1)/d$. Then $|x|$ is not a prime power (see Lemma \ref{l:ppower}) and $H = {\rm N}_{1}$ is the only maximal subgroup of $G$ containing such an element (see \cite[p.767]{GK}). Now, if $H = {\rm N}_{1}$ then any element of order $(q^6-1)/e$ is a derangement of non-prime power order. 

Suppose that $n=6$ and $q \in \{2,5\}$. The case $q=2$ can be checked directly, using \textsf{GAP} for example, so let us assume that $q=5$. Let $x \in G$ be an element of order $(5^6-1)/e = 434$. By inspecting the list of maximal subgroups of $G$ (see \cite[Tables 8.26 and 8.27]{BHR}), we deduce that $x$ is a derangement unless $H$ is of type ${\rm P}_{3}$, ${\rm GL}_{3}(5^2)$ or ${\rm GU}_{2}(5^3)$, so we may assume that $H$ is one of these subgroups, in which case any element of order $\ell_{10}(5)=521$ is a derangement. Suppose that 
$H = {\rm P}_{3}$. Fix an orthonormal basis for $V$ and let $y = {\rm diag}(1,1,\omega,\omega^{-1},\omega^{2},\omega^{-2})$ (modulo scalars), where $\omega \in \mathbb{F}_{25}$ is an element of order $6$. Then $y$ is a derangement. Similarly, if $H$ is of type ${\rm GL}_{3}(5^2)$ or ${\rm GU}_{2}(5^3)$, then any unipotent element with Jordan form $[J_2,J_1^4]$ is a derangement. This eliminates the case $G = {\rm U}_{6}(5)$.

A very similar argument applies if $n=4$. Here the cases $q \in \{2,3\}$ can be checked directly, so let us assume that $q \geqs 4$. Let $x \in G$ be an element of order $(q^3+1)/d$. Then $|x|$ is not a prime power (see Lemma \ref{l:ppower}) and again we reduce to the case $H = {\rm N}_{1}$ (see \cite[p.767]{GK}). We can now take any element of order $(q^4-1)/e$, which will be a  derangement of non-prime power order. 
\end{proof}

\vs

Next, we turn our attention to symplectic groups. Let $G_0 = {\rm PSp}_n(q)$ be a symplectic group with natural module $V$. As before, we will write ${\rm P}_{i}$ (respectively, ${\rm N}_{i}$) for the $G_0$-stabilizer of an $i$-dimensional totally isotropic (respectively, non-degenerate) subspace of $V$. We will also use $n = m \perp (n-m)$ to denote an orthogonal decomposition of $V$ of the form $V = V_1 \perp V_2$, where $V_1$ is a non-degenerate $m$-space. Further, we will say that a semisimple element $x \in G_0$ is of \emph{type} $m \perp (n-m)$ if it fixes such an orthogonal decomposition of $V$, acting irreducibly on $V_1$ and $V_2$. Similar notation is used in \cite{BGK,BTV,GK}.

To begin with, we will assume that $n \geqs 6$; the special case $G_0 = {\rm PSp}_{4}(q)$ will be handled separately in Lemma \ref{lem:Sp4}.

\begin{lemma}\label{lem:Sp}
Theorem \ref{t:main2} holds if $G_0 = {\rm PSp}_n(q)$ and $n \geqs 6$.
\end{lemma}

\begin{proof}
Set $d=(2,q-1)$ and write $n=2m$ with $m \geqs 3$. As before, we may assume that $G=G_0$.

\vs

\noindent \emph{Case 1.} $m$ odd. 

\vs

The case $(n,q) = (6,2)$ can be handled directly (using \textsf{GAP}, for example), so let us assume that $(n,q) \neq (6,2)$. Let $x \in G$ be an element of order $(q^{m}+1)/d$. If $q$ is even, then Lemma \ref{lem:primepower} implies that $|x|$ is not a prime power, and it is easy to see that the same conclusion also holds if $q$ is odd. By the main theorem of \cite{Ber}, $x$ is a derangement unless one of the following holds:
\begin{itemize}\addtolength{\itemsep}{0.2\baselineskip}
\item[{\rm (a)}] $H$ is a field extension subgroup of type ${\rm Sp}_{n/k}(q^k)$ for some prime divisor $k$ of $n$; 
\item[{\rm (b)}] $q$ is even and $H = O_{n}^{-}(q)$.
\end{itemize}
In (a), elements of order $\ell_{n-2}(q)$ are derangements, and so are unipotent elements with Jordan form $[J_2,J_1^{n-2}]$. Similarly, if (b) holds then elements of order $\ell_{m}(q)$ are derangements, and so are semisimple elements of type $(n-2) \perp 2$ and order 
$(q^{m-1}+1)(q+1)$.

\vs

\noindent \emph{Case 2.} $m \geqs 6$ even. 

\vs

First assume that $q$ is odd. Let $x \in G$ be a semisimple element of type $(n-4) \perp 4$, so  
$$|x| = \left\{\begin{array}{ll}
q^{m-2}+1 & \mbox{if $m \equiv 0 \imod{4}$} \\
(q^{m-2}+1)(q^{2}+1)/2 & \mbox{if $m \equiv 2 \imod{4}$.} 
\end{array}\right.$$
Clearly, if $m \equiv 2 \imod{4}$ then $|x|$ is divisible by $\ell_{4}(q)$ and $\ell_{n-4}(q)$, so $|x|$ is not a prime power. The same conclusion also holds if $m \equiv 0 \imod{4}$ (see Lemma \ref{lem:primepower}). By \cite[Proposition 5.10]{BGK}, we may assume that $H$ is of type ${\rm N}_{4}$ or ${\rm Sp}_{n/2}(q^2)$. In both cases, elements of order $\ell_{n-2}(q)$ are derangements. In addition, unipotent elements with Jordan form $[J_n]$ (respectively, $[J_2,J_1^{n-2}]$) are derangements if $H$ is of type ${\rm N}_{4}$ (respectively, ${\rm Sp}_{n/2}(q^2)$).  

Now assume that $q$ is even. Let $x \in G$ be a semisimple element of type $(n-2) \perp 2$ and order  $q^{m-1}+1$. Now Lemma \ref{lem:primepower} implies that $|x|$ is not a prime power, and by applying the main theorem of \cite{GPPS} we deduce that $x$ is a derangement unless $H \in \{{\rm N}_{2},O_{n}^{+}(q)\}$. In both cases, elements of order $\ell_{n}(q)$ are derangements. Also, unipotent elements with Jordan form $[J_n]$ are derangements if $H = {\rm N}_{2}$. Now, if $H=O_{n}^{+}(q)$ then let $y \in G$ be a block-diagonal element of the form $y=[y_1,y_2]$ (with respect to an orthogonal decomposition $n = (n-2) \perp 2$), where $y_1 \in {\rm Sp}_{n-2}(q)$ has order $\ell_{m-1}(q)$ and $y_2 \in {\rm Sp}_{2}(q)$ has order $q+1$. Then $y$ is a derangement and the result follows. 

\vs

\noindent \emph{Case 3.} $m = 4$.  

\vs

The case $q = 2$ can be checked directly, so we may assume that $q \geqs 3$. If $q$ is even, then we can repeat the relevant argument in Case 2. Now assume $q$ is odd. Let $x \in G$ be a semisimple element of type $6 \perp 2$ and order $q^3+1$.  By Lemma \ref{lem:primepower}, $|x|$ is not a prime power. The maximal subgroups of $G$ are listed in \cite[Tables 8.48 and 8.49]{BHR}, and we deduce that $x$ is a derangement unless $H$ is of type ${\rm N}_{2}$, ${\rm GU}_{4}(q)$ or ${\rm L}_{2}(q^3)$ (in the terminology of \cite{BHR,KL}, the latter possibility is an almost simple irreducibly embedded subgroup in the collection $\mathcal{S}$). In each of these exceptional cases, any element of order $(q^4+1)/2$ is a derangement. In addition, if $H={\rm N}_{2}$ then unipotent elements with Jordan form $[J_8]$ are derangements. Similarly, if $H$ is of type ${\rm GU}_{4}(q)$ or ${\rm L}_{2}(q^3)$ then elements with Jordan form $[J_2,J_1^6]$ are derangements.
\end{proof}

\begin{lemma}\label{lem:Sp4}
Theorem \ref{t:main2} holds if $G_0 = {\rm PSp}_n(q)$.
\end{lemma}

\begin{proof}
We may assume that $n=4$. The result can be checked directly if $q \leqs 7$, so let us assume that $q \geqs 8$. 

First assume that $q$ is odd. In terms of an orthogonal decomposition $4=2\perp 2$, let $x =[x_1,x_2] \in G$ (modulo scalars) be an element of order $p(q+1)$, where $x_1 \in {\rm Sp}_{2}(q)$ is a unipotent element of order $p$, and $x_2 \in {\rm Sp}_{2}(q)$ is irreducible of order $q+1$. By inspecting the list of maximal subgroups of $G$ (see \cite[Tables 8.12 and 8.13]{BHR}), we deduce that $x$ is a derangement unless $H$ is of type ${\rm P}_{1}$ or ${\rm Sp}_{2}(q) \wr {\rm S}_2$. In both of these cases, any element of order $\ell_4(q)$ is a derangement. Similarly, unipotent elements with Jordan form $[J_4]$ are derangements if $H$ is of type ${\rm Sp}_{2}(q) \wr {\rm S}_2$. Finally, suppose that $H = {\rm P}_{1}$. Now ${\rm Sp}_{2}(q)$ has precisely $\varphi(q+1)/2 \geqs 2$ distinct classes of elements of order $q+1$ (where $\varphi$ is the Euler totient function); if $y_1, y_2 \in {\rm Sp}_{2}(q)$ represent distinct classes, then $y = [y_1,y_2] \in G$ (modulo scalars) is a derangement since it does not fix a totally isotropic $1$-space. 

Now assume $q$ is even. As above, let $x \in G$ be an element of order $2(q+1)$. The maximal subgroups of $G$ are listed in \cite[Table 8.14]{BHR}, and we see that $x$ is a derangement unless $H$ is of type ${\rm P}_{1}$, ${\rm Sp}_{2}(q) \wr S_2 \cong O_{4}^{+}(q)$ or $O_{4}^{-}(q)$. For $H = {\rm P}_{1}$, we can repeat the argument in the $q$ odd case, so let us assume that $H = O_{4}^{\e}(q)$. If $\e=+$ then any element of order $\ell_4(q)$ is a derangement, and we can also find derangements of order $4$ (with Jordan form $[J_4]$), since there are two conjugacy classes of such elements in $G$, but only one in $H$. Finally, if $\e=-$ then we can find derangements of order $2$ (with Jordan form $[J_2^2]$; these are $a_2$-type involutions in the sense of Aschbacher and Seitz \cite{AS}), and also derangements of order $q+1$ of the form $[y_1,y_2]$ as above.
\end{proof}

\vs

To complete the proof of Theorem \ref{t:main2}, we may assume that $G_0 = {\rm P\Omega}_{n}^{\e}(q)$ is an orthogonal group, where $n \geqs 7$. The low-dimensional groups with $n \in \{7,8\}$ require special attention. We extend our earlier notation for orthogonal decompositions by writing $m^{\pm}$ to denote a non-degenerate $m$-space of type $\pm$ (when $m$ is even). Similarly, we write ${\rm N}_{m}^{\pm}$ for the $G_0$-stabilizer of such a subspace of the natural $G_0$-module $V$. If $q$ is even, we will also adopt the standard Aschbacher-Seitz notation for involutions (see \cite{AS}).

\begin{lemma}\label{lem:O7}
Theorem \ref{t:main2} holds if $G_0 = \Omega_7(q)$.
\end{lemma}

\begin{proof}
We may assume that $G=G_0$. The case $q=3$ can be checked directly, so we may assume that $q \geqs 5$ (recall that $q$ is odd). Let $x \in G$ be an element of order $(q^3+1)/2$, which is not a prime power. By \cite[Proposition 5.20]{BGK}, $x$ is a derangement unless $H = {\rm N}_{6}^{-}$, in which case any element of order $\ell_3(q)$ is a derangement, and so are unipotent elements with Jordan form $[J_7]$.
\end{proof}

\begin{lemma}\label{lem:O8+}
Theorem \ref{t:main2} holds if $G_0 = {\rm P\Omega}_8^{+}(q)$.
\end{lemma}

\begin{proof}
As usual, we may assume that $G=G_0$. Let $V$ be the natural module for $G_0$. The case $q=2$ can be checked directly, using \textsf{GAP}. Next suppose that $q=3$. Let $x \in G$ be an element of order $20$, fixing a decomposition of $V$ of the form $8 = 4^{-}\perp 4^{-}$. As indicated in \cite[Table 3]{BGK}, $x$ is a derangement unless the type of $H$ is one of the following:
$${\rm P}_{4}, O_7(3), O_{4}^{-}(3) \wr {\rm S}_2, {\rm GU}_{4}(3), {\rm Sp}_{4}(3) \otimes {\rm Sp}_{2}(3)$$
where $O_7(3)$ is irreducible and ${\rm P}_{4}$ is the stabilizer in $G$ of a maximal totally singular subspace of $V$.

By considering elements of order $14$, we can immediately eliminate the cases ${\rm P}_{4}$, $O_{4}^{-}(3) \wr S_2$ and ${\rm Sp}_{4}(3) \otimes {\rm Sp}_{2}(3)$. Similarly, $G$ contains derangements of order $15$ if $H$ is of type ${\rm GU}_{4}(3)$. Finally, suppose that $H$ is an irreducible subgroup of type $O_7(3)$. To see that \eqref{e:star} does not hold, we may replace $H$ by a conjugate $H^{\tau}$, where $\tau \in {\rm Aut}(G)$ is an appropriate triality graph automorphism such that $H^{\tau}$ is the stabilizer in $G$ of a non-degenerate $1$-space. For this reducible subgroup, elements with Jordan form $[J_2^4]$ are derangements, and so are elements $y \in G$ of order $5$ of the form $y = \hat{y}Z$, where $Z = Z(\Omega_{8}^{+}(3))$ and $C_V(\hat{y})$ is trivial (the eigenvalues of $\hat{y}$ (in $\mathbb{F}_{3^4}$) are the nontrivial fifth roots of unity, each occurring with multiplicity $2$). 

For the remainder, we may assume that $q \geqs 4$. Let $x \in G$ be an element of order 
$(q^3+1)/(2,q-1)$, fixing an orthogonal decomposition $8 = 6^{-} \perp 2^{-}$. Then $|x|$ is not a prime power, and $x$ is a derangement unless $H$ is of type ${\rm N}_{2}^{-}$ or ${\rm GU}_{4}(q)$ (see \cite[p.767]{GK}). In both of these cases, elements of order $\ell_3(q)$ are derangements. Similarly, if $q$ is odd then unipotent elements with Jordan form $[J_7,J_1]$ are also derangements. Finally, if $q$ is even and $H$ is of type ${\rm N}_{2}^{-}$ (respectively, 
${\rm GU}_{4}(q)$) then unipotent elements with Jordan form $[J_4^2]$ (respectively, $[J_2^2,J_1^4]$; $c_2$-involutions in the terminology of \cite{AS}) are derangements. The result follows.
\end{proof}

\begin{lemma}\label{lem:O8-}
Theorem \ref{t:main2} holds if $G_0 = {\rm P\Omega}_8^{-}(q)$.
\end{lemma}

\begin{proof}
Again, we may assume that $G=G_0$. If $q \leqs 3$ then we can use \textsf{GAP} to verify the result, so let us assume that $q \geqs 4$. The maximal subgroups of $G$ are listed in \cite[Tables 8.52 and 8.53]{BHR}. By considering elements of order $\ell_{8}(q)$ and $\ell_{6}(q)$, we can eliminate subfield subgroups, together with the reducible subgroups of type ${\rm P}_{2}$, ${\rm P}_{3}$, ${\rm N}_{2}^{-}$, ${\rm N}_{3}$ and ${\rm N}_{4}^{+}$. Similarly, elements of order $\ell_{8}(q)$ and $\ell_{4}(q)$ are derangements if $H$ is a non-geometric subgroup of type ${\rm L}_{3}^{\e}(q)$. Therefore, to complete the proof, we may assume that $H$ is either a field extension subgroup of type $O_{4}^{-}(q^2)$, or a reducible subgroup of type ${\rm P}_{1}$, ${\rm N}_{2}^{+}$, $O_7(q)$ ($q$ odd) or ${\rm Sp}_{6}(q)$ ($q$ even).

If $H$ is of type $O_{4}^{-}(q^2)$, then elements of order $\ell_{6}(q)$ are derangements, as well as unipotent elements with Jordan form $[J_{7},J_{1}]$ if $q$ is odd, and unipotent elements with Jordan form $[J_2^2,J_1^4]$ ($a_2$-type involutions) if $q$ is even. Similarly, if $H = {\rm P}_{1}$ or ${\rm N}_{2}^{+}$ then elements of order $\ell_{8}(q)$ and $\ell_{3}(q)$ are derangements (note that an element of order $\ell_3(q)$ fixes a $2^{-}$-space, but not a $2^{+}$-space). Finally, suppose $H$ is of type $O_7(q)$ ($q$ odd) or ${\rm Sp}_{6}(q)$ ($q$ even). In both cases, elements of order $\ell_8(q)$ are derangements. In addition, there are derangements with Jordan form $[J_5,J_3]$ ($q$ odd) and $[J_4^2]$ ($q$ even). 
\end{proof}

\begin{lemma}\label{lem:O}
Theorem \ref{t:main2} holds if $G_0 = {\rm P\Omega}_n^{\e}(q)$.
\end{lemma}

\begin{proof}
We may assume that $G=G_0$ and $n \geqs 9$. We have three cases to consider.

\vs

\noindent \emph{Case 1.} $G_0 = {\rm P\Omega}_{n}^{+}(q)$ and $n \geqs 10$.

\vs

Write $n=2m$ and first assume that $m$ is odd. Let $x \in G$ be an element of order $(q^{(m-1)/2}+1)(q^{(m+1)/2}+1)/(4,q-1)$, fixing an orthogonal decomposition of the form $(m+1)^{-} \perp (m-1)^{-}$. Then Lemma \ref{lem:numerical} implies that $|x|$ is not a prime power, so by \cite[Proposition 5.13]{BGK} we may assume that $H = {\rm N}_{m-1}^{-}$. In this situation, elements of order $\ell_{n-2}(q)$ are derangements, and so are unipotent elements with Jordan form $[J_{n-1},J_{1}]$ ($q$ odd) or $[J_{n-2},J_2]$ ($q$ even).

A similar argument applies if $m$ is even. Here we take an element $x \in G$ of order 
$(q^{(m-2)/2}+1)(q^{(m+2)/2}+1)/(4,q-1)$, fixing a decomposition $(m+2)^{-} \perp (m-2)^{-}$.
Then $|x|$ is not a prime power, and \cite[Proposition 5.14]{BGK} implies that $x$ is a derangement unless $H$ is of type ${\rm N}_{m-2}^{-}$ or $O_{n/2}^{+}(q^2)$. In the former case, we complete the argument as above, so let us assume that $H$ is of type $O_{n/2}^{+}(q^2)$. Any element of order $\ell_{n-2}(q)$ is a derangement, and so are unipotent elements with Jordan form $[J_{n-1},J_{1}]$ if $q$ is odd. Finally, if $q$ is even then $a_2$-type involutions are derangements.

\vs

\noindent \emph{Case 2.} $G_0 = {\rm P\Omega}_{n}^{-}(q)$ and $n \geqs 10$.

\vs

Again, write $n=2m$. First assume that $m \geqs 11$. Let $x \in G$ be an element of order 
$${\rm lcm}(q^{m-5}+1,q^3+1,q^2+1)/(2,q-1)$$
fixing a decomposition $(n-10)^{-} \perp 6^{-} \perp 4^{-}$. Then $|x|$ is not a prime power,  and \cite[Proposition 5.16]{BGK} implies that $x$ is a derangement unless $H$ is of type ${\rm N}_{4}^{-}$, ${\rm N}_{6}^{-}$ or ${\rm N}_{10}^{+}$. In each of these cases, it is clear that elements of order $\ell_{n}(q)$ and $\ell_{n-2}(q)$ are derangements.

Next suppose that $m \in \{5,6,7,9,10\}$. Let $x \in G$ be an irreducible element of order 
$(q^m+1)/(2,q-1)$. We claim that $|x|$ is not a prime power (here we require $m \neq 8$). If $q$ is even, this follows immediately from Lemma \ref{lem:primepower}, so let us assume that $q$ is odd. Suppose $m=5$ and $q^5+1 = 2r^e$ for some prime $r$ and positive integer $e$. Then $(q+1)(q^4-q^3+q^2-q+1) = 2r^e$ and $r = \ell_{10}(q)$. Therefore, $q+1=2$ is the only possibility, which is absurd. Similarly, if $m=6$ and $q^6+1 = (q^2+1)(q^4-q^2+1) = 2r^e$, then $r=\ell_{12}(q)$ and $q^2+1=2$, which is not possible. The other cases are entirely similar. Now, by the main theorem of \cite{Ber}, $x$ is a derangement unless $H$ is a field extension subgroup of type $O_{n/k}^{-}(q^k)$ ($k$ a prime divisor of $n$, $n/k \geqs 4$ even) or ${\rm GU}_{n/2}(q)$ ($n/2$ odd). In both cases, elements of order $\ell_{n-2}(q)$ are derangements. In addition, there are unipotent derangements; take $[J_{n-1},J_1]$ if $q$ is odd, an $a_2$-involution if $q$ is even and $H$ is of type $O_{n/k}^{-}(q^k)$, and a $c_2$-involution if $q$ is even and $H$ is of type ${\rm GU}_{n/2}(q)$.

Finally, let us assume that $m=8$.  As in \cite[Table II]{GK}, let $x \in G$ be an element of order ${\rm lcm}(q^5+1,q^2+1,q+1)/(2,q-1)$, fixing an orthogonal decomposition of the form $10^{-} \perp 4^{-} \perp 2^{-}$. Note that $|x|$ is divisible by $\ell_{10}(q)$ and $\ell_4(q)$, so it is not a prime power. As indicated in \cite[Table II]{GK}, $x$ is a derangement unless $H$ is of type ${\rm N}_{2}^{-}$, ${\rm N}_{4}^{-}$ or ${\rm N}_{6}^{+}$. In each of these cases, elements of order $\ell_{16}(q)$ and $\ell_{14}(q)$ are derangements.

\vs

\noindent \emph{Case 3.} $G_0 = \Omega_{n}(q)$ and $n \geqs 9$ is odd.

\vs

Write $n=2m+1$ and note that $q$ is odd. First assume $m$ is odd. Let $x \in G$ be an element of order 
$${\rm lcm}(q^{(m+1)/2}+1, q^{(m-1)/2}+1)/2 = (q^{(m+1)/2}+1)(q^{(m-1)/2}+1)/4,$$ 
fixing an orthogonal decomposition $(m+1)^{-} \perp (m-1)^{-} \perp 1$. Note that $|x|$ is divisible by $\ell_{m+1}(q)$ and $\ell_{m-1}(q)$, so $|x|$ is not a prime power. Let $H$ be a
maximal subgroup of $G$ containing $x$. By carefully applying the main theorem of \cite{GPPS}, we deduce that $H \in \{{\rm N}_{m+1}^{-}, {\rm N}_{m-1}^{-}, {\rm N}_{2m}^{+}\}$. For example, the order of $x$ rules out subfield subgroups and imprimitive subgroups of type $O_1(q) \wr \SSS_{n}$ (see \cite[Remark 5.1(i)]{BGK}), and the dimensions of the irreducible constituents of $x$ are incompatible with field extension subgroups of type $O_{n/k}(q^k)$. Now, if $H$ is one of these reducible subgroups, then elements of order $\ell_{n-1}(q)$ and $\ell_{n-3}(q)$ are derangements. The result follows.

A similar argument applies if $m$ is even. Here we take $x \in G$ to be an element of order 
$${\rm lcm}(q^{(m+2)/2}+1, q^{(m-2)/2}+1)/2 = (q^{(m+2)/2}+1)(q^{(m-2)/2}+1)/4,$$ 
fixing a decomposition $(m+2)^{-} \perp (m-2)^{-} \perp 1$. We claim that 
$|x|$ is not a prime power. This is clear if $m \geqs 6$, or if $m=4$ and $q$ is not a Mersenne prime, since $|x|$ is divisible by $\ell_{m \pm 2}(q)$. Suppose that $m=4$ and $q$ is a Mersenne prime. If $q=3$ then $|x|=28$ and the claim holds, and if $q>3$ then $|x|$ is divisible by $2$ and $\ell_6(q)$. This justifies the claim.
Using \cite{GPPS} one can check that the only maximal subgroups of $G$ containing $x$ are of type ${\rm N}_{m+2}^{-}$, ${\rm N}_{m-2}^{-}$ or ${\rm N}_{2m}^{+}$, so we may assume that $H$ is one of these subgroups. Here we observe that elements of order $\ell_{n-1}(q)$ are derangements, and so are unipotent elements with Jordan form $[J_n]$.
\end{proof}

\vs

This completes the proof of Theorem \ref{t:main2}.

\section{Affine groups}\label{s:affine}

Let $G$ be a finite primitive permutation group. By Theorem \ref{t:main1}, if \eqref{e:star} holds then $G$ is either almost simple or affine. In the previous section, we determined all the almost simple examples, and we now turn our attention to the affine groups with property \eqref{e:star}. Our main aim is to prove Theorem \ref{t:main3}.  

Let $G = HV \leqs {\rm AGL}(V)$ be a finite affine primitive permutation group with point stabilizer $H = G_0$ and socle $V = (\mathbb{Z}_{p})^k$. As an abstract group, $G$ is a semidirect product of $V$ by $H$. Therefore, we will begin our analysis by studying the structure of a general semidirect product $G=H \ltimes N$ with property \eqref{e:star}, so $G$ is a finite group, $H$ is a proper subgroup and $N$ is a normal subgroup of $G$ such that $G = HN$ and $H \cap N = 1$. 

We will need some additional notation. If $K$ is a subgroup of $G$ and $g\in G$, then we set
$$[K,g] = \{ [k,g]=k^{-1}g^{-1}kg \,:\, k \in K\}.$$ 
We also write $K^*$ for the set of all nontrivial elements of $K$.

\begin{lemma}\label{lem:general facts}
Let $G=H \ltimes N$. The following hold:
\begin{itemize}\addtolength{\itemsep}{0.2\baselineskip}
\item[{\rm (i)}] $\Centralizer_G(x)=\Centralizer_H(x)\Centralizer_N(x)$ for all $x\in H$.
\item[{\rm (ii)}] If $K\leqs H$, then $K\cap K^n=\Centralizer_K(n)$ for all $n \in N^*$.
\item[{\rm (iii)}] If $N$ is abelian, then $\Delta_H(G)=\{tv \, : \, t\in H,v\in N\setminus [N,t]\}$.
\item[{\rm (iv)}] If property \eqref{e:star} holds, then $N$ is an $r$-group.
\end{itemize} 
\end{lemma}

\begin{proof} 
First consider part (i). The result is clear if $x=1$, so assume that $x\in H^*$. The inclusion 
$\Centralizer_H(x)\Centralizer_N(x) \subseteq \Centralizer_G(x)$ is clear. Conversely, suppose that $g=hn\in\Centralizer_G(x)$ where $h\in H$, $n\in N$. Then $hnx=xhn$. Multiplying both sides by $(xh)^{-1}=h^{-1}x^{-1}$, we obtain 
\[hnxh^{-1}x^{-1}=(xh)n(xh)^{-1}\]
which implies that 
\[(hnh^{-1})(hxh^{-1}x^{-1})=(xh)n(xh)^{-1}.\] 
Since $n\in N\normeq G$ and $h,x\in H$, we deduce that 
$$hxh^{-1}x^{-1}=(hn^{-1}h^{-1})(xh)n(xh)^{-1}\in H\cap N=1$$ 
so $h\in \Centralizer_H(x)$. Since $hnx=xhn=hxn$, we deduce that $nx=xn$ and thus $n\in\Centralizer_N(x)$. Therefore, $g=hn\in\Centralizer_H(x)\Centralizer_N(x)$ and part (i) follows.

For part (ii), let $K\leqs H$ and let $n\in N^*$. Assume that $y\in K\cap K^n$. Then $y=k^n\in K$ for some $k\in K$, so 
$$k^{-1}y=k^{-1}n^{-1}kn=(k^{-1}n^{-1}k)n\in K\cap N=1,$$ 
which implies that $kn=nk$ and $y=k$, or equivalently $y \in \Centralizer_K(n)$. Therefore, 
$K\cap K^n \leqs \Centralizer_K(n)$. Conversely, if $y\in \Centralizer_K(n)$ then $y\in K$ and $y=n^{-1}yn\in K^n$, so $y\in K\cap K^n$ and thus $\Centralizer_K(n)\leqs K\cap K^n$. The result follows.

Now consider part (iii). Assume that $N$ is abelian. Set
$$\Gamma:= \{tv \, : \, t\in H,v\in N\setminus [N,t]\}.$$ 
First we claim that $\Gamma \subseteq \Delta_H(G)$. Let $g\in \Gamma$, say $g=hn$ with $h\in H$ and $n\in N\setminus [N,h]$. Seeking a contradiction, suppose that $g \not\in \Delta_H(G)$. Then $g\in H^t$ for some $t\in G$. Since $t\in G = HN$, we may write $t=h_1m_1$ with $h_1\in H$ and $m_1\in N$. It follows that $g\in H^t=H^{m_1}$, so $m_1gm_1^{-1}\in H$. Let $m:=m_1^{-1}\in N$. Then $m^{-1}gm=m^{-1}hnm=h^mn\in H$ (note that $nm=mn$ since $N$ is abelian) and thus $h^{-1}h^mn=[h,m]n\in H$. We also have $[h,m]n=(h^{-1}m^{-1}h)mn\in N$, so $[h,m]n\in H\cap N=1$ and we deduce that $n=[m,h]\in [N,h]$, contradicting our choice of $n$. We have now shown that $\Gamma \subseteq \Delta_H(G)$. Conversely, suppose that 
$g=hn\in \Delta_H(G)$ with $h\in H$, $n\in N$. We claim that $n \in N\setminus [N,h]$. Seeking a contradiction, suppose that $n \in [N,h]$, say $n=[m,h]$ for some $m\in N$. Then 
$m^{-1}(hn)m=h$, or equivalently $g^m\in H$, which is a contradiction.

Finally, let us turn to part (iv). If $x\in N^*$ then $x^G\subset N$, so $x^G\cap H \subseteq N \cap H=1$ and thus $x^G\cap H=\emptyset$ since $x\neq 1$. Therefore $N^*\subseteq \Delta_H(G)$. In particular, if every element of $\Delta_H(G)$ is an $r$-element (for some fixed prime $r$), then every element of $N$ is also an $r$-element and thus $N$ is an $r$-group. 
\end{proof}

\begin{lemma}\label{lem:equivalence}
Let $G=H\ltimes N$, where $N$ is an $r$-group for some prime $r$. Then the following are equivalent:
\begin{itemize}\addtolength{\itemsep}{0.2\baselineskip}
\item[{\rm (i)}] Property \eqref{e:star} holds. 
\item[{\rm (ii)}] $\Centralizer_H(n)=H\cap H^n$ is an $r$-group for all $n\in N^*$.
\item[{\rm (iii)}]  $\Centralizer_N(x)=1$ for every nontrivial $r'$-element $x \in H$. In other 
words, every nontrivial $r'$-element of $H$ induces a fixed-point-free automorphism of $N$ via conjugation.
\end{itemize}
\end{lemma}

\begin{proof}
First we will show that (i) implies (ii). Suppose that \eqref{e:star} holds. Let $n\in N^*$. We claim that $\Centralizer_H(n)$ is an $r$-group. Notice that $\Centralizer_H(n)=H\cap H^n$ by Lemma \ref{lem:general facts}(ii). Seeking a contradiction, suppose that $|\Centralizer_H(n)|$ is divisible by a prime $s \neq r$. Choose $y\in \Centralizer_H(n)$ with $|y|=s$ and let $g:=ny=yn\in G$. We claim that $g\in\Delta_H(G)$, which would be a contradiction since $|g|=|n|s$ is not a power of $r$. Assume that $g \not\in\Delta_H(G)$, so $g\in H^t$ for some $t\in G$. Since $G=HN$, we may assume that $t \in N$. Then 
$$g^s=(ny)^s=n^sy^s=n^s\in H^t$$ 
and $n^s\in N\normeq G$, so $t(n^s)t^{-1}\in H\cap N=1$ and thus $n^s=1$, which is not possible since $n$ is a nontrivial $r$-element. Therefore, $g=ny\in\Delta_H(G)$ as required. 

Next we will show that (ii) implies (i). Suppose that  $\Centralizer_H(n)$ is an $r$-group for all  $n\in N^*$. Let $g \in \Delta_H(G)$, say $g=hn$ with $h\in H$ and $n\in N^*$. We claim that $g$ is an $r$-element. Seeking a contradiction, suppose that $m:=|g|$ is divisible by a prime $s \neq r$. Set $x:=g^{m/s} \in G$ and let $S$ be a Sylow $s$-subgroup of $H$.  Then $|x|=s$ and $S$ is also a Sylow $s$-subgroup of $G$ since $|G:H|=|N|$ is coprime to $s$. By Sylow's theorem, $x^t \in S\leqs H$ for some $t \in G$. Since $g^G \subseteq \Delta_H(G)$, replacing $g$ by $g^{t^{-1}}$ we may assume that $x \in H$. Then $g \in \Centralizer_G(x)=\Centralizer_H(x)\Centralizer_N(x)$ by Lemma \ref{lem:general facts}(i). 

Suppose that $\Centralizer_N(x) \neq 1$, say $1 \neq n \in \Centralizer_N(x)$. Then 
$x \in \Centralizer_H(n)$, but this is a contradiction since $|x|=s$ and we are assuming that $\Centralizer_H(n)$ is an $r$-group. Therefore, $\Centralizer_N(x)=1$ and thus 
$\Centralizer_G(x)=\Centralizer_H(x)$. Hence $g\in \Centralizer_H(x)\leqs H$, which contradicts the fact that $g\in \Delta_H(G)$. This final contradiction shows that $g$ is an $r$-element, so \eqref{e:star} holds.

Now let us show that (ii) implies (iii). Suppose that $\Centralizer_H(n)$ is an $r$-group for all 
$n\in N^*$. Let $x\in H^*$ be an $r'$-element. We claim that $\Centralizer_N(x)=1$. Seeking a  contradiction, suppose that $1 \neq n \in \Centralizer_N(x)$. Then 
$x\in \Centralizer_H(n)$, so $|\Centralizer_H(n)|$ is divisible by $|x|$, which is not an $r$-power. This contradicts the assumption that $\Centralizer_H(n)$ is an $r$-group.

To complete the proof, it remains to show that (iii) implies (ii). Suppose that 
$\Centralizer_N(x)=1$ for every nontrivial $r'$-element $x \in H$. Let $n \in N^*$. If $\Centralizer_H(n)$ is not an $r$-group, then there exists an element $x\in \Centralizer_H(n)$ with $|x|=s$, where $s \neq r$ is a prime. Therefore, $1\neq n \in \Centralizer_N(x)$, which is not possible since $\Centralizer_N(x)=1$. 
\end{proof}

We are now in a position to prove Theorem \ref{t:main3}. 

\begin{proof}[Proof of Theorem \ref{t:main3}]
Let $G = HV \leqs {\rm AGL}(V)$ be a finite affine primitive permutation group with point stabilizer $H=G_0$ and socle $V = (\Z_p)^k$, where $p$ is a prime and $k \geqs 1$. If property \eqref{e:star} holds, then $r=p$ and Lemma \ref{lem:equivalence} implies that no nontrivial $r'$-element of $H$ has fixed points on $V\setminus\{0\}$. Therefore, the pair $(H,V)$ is $r'$-semiregular in the sense of \cite{FLT}. Conversely, if $r=p$ and $(H,V)$ is $r'$-semiregular, then $\Centralizer_V(x)=0$ for every nontrivial $r^\prime$-element 
$x\in H$, so Lemma \ref{lem:equivalence} implies that $G$ has property \eqref{e:star}. \end{proof}

\begin{rem}\label{r:affine0}
\emph{Note that the equivalence of (i) and (ii) in Lemma \ref{lem:equivalence} implies that an affine group $G = HV \leqs {\rm AGL}(V)$ has property \eqref{e:star} if and only if every two-point stabilizer in $G$ is an $r$-group.}
\end{rem}

If $G = HV \leqs {\rm AGL}(V)$ is an affine group (with $V = (\Z_r)^k$) and $r \not\in \pi(H)$, then $G$ is a Frobenius group and property \eqref{e:star} clearly holds. Therefore, we may focus on the case where $r \in \pi(H)$. As noted in the Introduction, detailed information on $r'$-semiregular pairs $(H,V)$ was initially obtained by Guralnick and Wiegand in \cite[Section 4]{GW}, where this notion arises naturally in their study of the multiplicative structure of field extensions. Similar results were established in later work by Fleischmann et al. \cite{FLT}. In both papers, the main aim is to determine the structure of $H$. For solvable affine groups, we have the following result (in the statement, $\OO_{r'}(Y)$ denotes the largest normal $r'$-subgroup of $Y$):

\begin{proposition}\label{p:flt1}
Let $G = HV \leqs {\rm AGL}(V)$ be a finite affine primitive permutation group with point stabilizer $H = G_0$ and socle $V = (\Z_r)^k$. Assume that $H$ is solvable and $r \in \pi(H)$. Then $G$ has property \eqref{e:star} only if $H  \cong X \times Y$ or $(X \times Y){:}2$, where $X \in \{1,{\rm SL}_{2}(3)\}$, $Y = \OO_{r'}(Y)R$ and $R$ is a Sylow $r$-subgroup of $Y$.
\end{proposition}

\begin{proof}
This follows from \cite[Theorem 2.1]{FLT}.
\end{proof}

The main result for a perfect group $H$ is Proposition \ref{p:flt2} below (see \cite[Theorem 4.1]{FLT}; also see \cite[Theorem 4.2]{GW}). In part (iv), $\mathcal{S} = \{5,13,37,73, \ldots\}$ is the set of all primes $s$ satisfying the following conditions:
\begin{itemize}\addtolength{\itemsep}{0.2\baselineskip}
\item[{\rm (a)}] $s = 2^a3^b+1$, where $a \geqs 2$ and $b \geqs 0$;
\item[{\rm (b)}] $(s+1)/2$ is a prime.
\end{itemize} 
It is not known whether or not $\mathcal{S}$ is finite.

\begin{proposition}\label{p:flt2}
Let $G= HV \leqs {\rm AGL}(V)$ be a finite affine primitive permutation group with point stabilizer $H=G_0$ and socle $V = (\Z_r)^k$. Assume that $H$ is perfect and $r \in \pi(H)$. Then $G$ has property \eqref{e:star} only if one of the following holds:
\begin{itemize}\addtolength{\itemsep}{0.2\baselineskip}
\item[{\rm (i)}] $H  \cong {\rm SL}_{2}(r^a)$, where $a \geqs 1$ and $r^a>3$;
\item[{\rm (ii)}] $H \cong {}^2{\rm B}_2(2^{2a+1})$, $r=2$ and $a \geqs 1$;
\item[{\rm (iii)}] $H \cong {}^2{\rm B}_2(2^{2a+1}) \times {\rm SL}_{2}(2^{2b+1})$, $r=2$ and $a,b \geqs 1$ such that $(2a+1,2b+1) = 1$;
\item[{\rm (iv)}] $H \cong {\rm SL}_{2}(s)$, $r=3$ and $s \in \mathcal{S} \cup \{7,17\}$.
\end{itemize}
\end{proposition}

For instance, $H = {\rm SL}_{2}(7)$ has a $12$-dimensional faithful, irreducible module $V$ over $\F_{3}$, and the corresponding affine group $G = HV$ has property \eqref{e:star} (with $\Ed(G) = \{3,9\}$). In the general case, we refer the reader to \cite[Theorem 6.1]{FLT} for a detailed description of the structure of $H$.

Finally, let us suppose that $G= HV \leqs {\rm AGL}(V)$ is a finite affine primitive permutation group with property \eqref{e:star}. Set
$$\Ed(G)=\Ed_H(G)=\{|x|\,:\,x\in\Delta_H(G)\}.$$
Can we determine when $\Ed(G) = \{r\}$? In order to address this question, let $P$ be a Sylow $r$-subgroup of $G$. Then $V\leqs P$ since $V$ is a normal $r$-subgroup of $G$, and we have $P=(H\cap P)V=KV$ with $K:=H\cap P$. Note that $P=KV$ is a semidirect product.

\begin{proposition}\label{p:prime}
Let $G = HV \leqs {\rm AGL}(V)$ be a finite affine primitive permutation group with point stabilizer $H = G_0$ and socle $V = (\Z_r)^k$. Assume that property \eqref{e:star} holds. Let $P$ be a Sylow $r$-subgroup of $G$ and set $K = H \cap P$. Then the following hold:
\begin{itemize}\addtolength{\itemsep}{0.2\baselineskip}
\item[{\rm (i)}]  $P=KV$ is a transitive permutation group on $P/K$.
\item[{\rm (ii)}] $\Delta(G)=\bigcup_{g\in G}\Delta_K(P)^g$ and $\Ed(G)=\Ed_K(P)$.
\end{itemize} 
\end{proposition}

\begin{proof} 
As above, $P=KV$ is a semidirect product. For part (i), it suffices to show that the core $L$ of $K$ in $P$ is trivial. We have $L\leqs K\leqs H$ and $L\normeq P$, so $[L,V]\leqs L\cap V\leqs K\cap V=1$ and thus $L\leqs \Centralizer_K(V)\leqs \Centralizer_H(V)=1$ (here we are using the fact that $V$ is a faithful irreducible $H$-module). This proves (i).

Now consider part (ii). Clearly, it suffices to show that the first equality holds. By applying Lemma \ref{lem:general facts}(iii) we have 
\[\Delta_K(P)=\{tv \, : \, t\in K,v\in V\setminus [V,t]\}.\] 
Since $K\leqs H$, a further application of Lemma \ref{lem:general facts}(iii) (this time for $G=HV$) shows that $\Delta_K(P)\subseteq \Delta(G)$. As $\Delta(G)$ is a normal subset of $G$, it follows that 
\[\bigcup_{g\in G}\Delta_K(P)^g\subseteq\Delta(G).\] 

Since property \eqref{e:star} holds, every $g\in \Delta(G)$ is an $r$-element, so some 
$G$-conjugate of $g$ is in $P$. Without loss of generality, we may assume that $g\in P=KV$. By Lemma \ref{lem:general facts}(iii) we have $g=hn$, with $h\in H$ and $n\in V\setminus [V,h]$. Moreover, since $V\leqs P$ and $g\in P$, we have $h=gn^{-1}\in H\cap P=K$. Therefore, by applying Lemma \ref{lem:general facts}(iii) once again, we conclude that $g=hn \in \Delta_K(P)$, so $\Delta(G)=\bigcup_{g\in G}\Delta_K(P)^g$ and the proof is complete. 
\end{proof}

Now, if we assume that $G=HV$ has property \eqref{e:star}, then part (ii) of Proposition \ref{p:prime} implies that $\Ed(G)=\{r\}$ if and only if $\Ed_K(P)=\{r\}$. Clearly, if $P$ has exponent $r$, then $\Ed_K(P)=\{r\}$. Conversely, if $\Ed_K(P)=\{r\}$ with $r=2$ or $3$, then a theorem of Mann and Praeger \cite[Proposition 2]{MP} implies that $P$ has exponent $r$. 
In fact, for this specific transitive group $P$ we can show that the same conclusion holds for \emph{any} prime $r$ (we thank an anonymous referee for pointing this out).

\begin{theorem}\label{c:prime}
Let $G = HV \leqs {\rm AGL}(V)$ be a finite affine primitive permutation group with point stabilizer $H = G_0$ and socle $V = (\mathbb{Z}_{p})^k$, where $p$ is a prime and $k \geqs 1$. Then every derangement in $G$ has order $r$, for some fixed prime $r$, if and only if $r=p$ and the following two conditions hold:
\begin{itemize}\addtolength{\itemsep}{0.2\baselineskip}
\item[{\rm (i)}] Every two-point stabilizer in $G$ is an $r$-group;
\item[{\rm (ii)}] A Sylow $r$-subgroup of $G$ has exponent $r$.
\end{itemize}
\end{theorem}

\begin{proof}
Let $P$ be a Sylow $r$-subgroup of $G$. First assume that $r=p$ and (i) and (ii) hold. By (i), the pair $(H,V)$ is $r'$-semiregular so Theorem \ref{t:main3} implies that property \eqref{e:star} holds. Therefore, $\Ed(G)=\Ed_K(P)$ by Proposition \ref{p:prime}(ii) (with $K = H \cap P$) and thus condition (ii) implies that $\Ed(G)=\{r\}$ as required.

Conversely, let us assume that $\Ed(G)=\{r\}$, so $r=p$ and property \eqref{e:star} holds. By Theorem \ref{t:main3}, every two-point stabilizer in $G$ is an $r$-group and so it remains to show that $P$ has exponent $r$. Seeking a contradiction, suppose that ${\rm exp}(P) \geqs r^2$. Note that $r$ divides $|H|$. Let $Q$ be a Sylow $r$-subgroup of $H$. Let $x \in P$ be an element of order $r^2$ and observe that $x$ belongs to a conjugate of $H$ (since  $\Ed(G)=\{r\}$), so ${\rm exp}(Q) \geqs r^2$. We may assume $x \in H$ and we choose an element $v \in V \setminus [V,x]$. Then $xv \in G$ is a derangement by Lemma \ref{lem:general facts}(iii), but $|xv| \geqs r^2$ so we have reached a contradiction.
\end{proof}

\end{document}